\documentclass{article}

\usepackage[nonatbib, preprint]{neurips_2020}

\usepackage[utf8]{inputenc} 
\usepackage[T1]{fontenc}    
\usepackage{hyperref}       
\usepackage{url}            
\usepackage{booktabs}       
\usepackage{amsfonts}       
\usepackage{nicefrac}       
\usepackage{microtype}      
\usepackage{amsmath}
\usepackage{xr}
\usepackage{subcaption}
\usepackage{amssymb}
\usepackage{todonotes}
\usepackage{times}
\usepackage{graphicx}
\usepackage{url}
\usepackage{slantsc}
\usepackage{amsthm}
\usepackage{xcolor}

\newtheorem{theorem}{Theorem}

\title{Flexible mean field variational inference using mixtures of non-overlapping exponential families}
\author{%
    Jeffrey P. Spence \\
  Stanford University\\
  Stanford, CA 94305 \\
  \texttt{jspence@stanford.edu} \\
}

\begin{document}

\maketitle

\

\begin{abstract} 
Sparse models are desirable for many applications across diverse domains as they can perform automatic variable selection, aid interpretability, and provide regularization.  When fitting sparse models in a Bayesian framework, however, analytically obtaining a posterior distribution over the parameters of interest is intractable for all but the simplest cases.  As a result practitioners must rely on either sampling algorithms such as Markov chain Monte Carlo or variational methods to obtain an approximate posterior.  Mean field variational inference is a particularly simple and popular framework that is often amenable to analytically deriving closed-form parameter updates.  When all distributions in the model are members of exponential families and are conditionally conjugate, optimization schemes can often be derived by hand.  Yet, I show that using standard mean field variational inference can fail to produce sensible results for models with sparsity-inducing priors, such as the spike-and-slab.  Fortunately, such pathological behavior can be remedied as I show that mixtures of exponential family distributions with non-overlapping support form an exponential family.  In particular, any mixture of a diffuse exponential family and a point mass at zero to model sparsity forms an exponential family.  Furthermore, specific choices of these distributions maintain conditional conjugacy.  I use two applications to motivate these results: one from statistical genetics that has connections to generalized least squares with a spike-and-slab prior on the regression coefficients; and sparse probabilistic principal component analysis.  The theoretical results presented here are broadly applicable beyond these two examples.
\end{abstract}

\section{Introduction}
Bayesian graphical models \cite{lauritzen1988local} are widely used across vast swaths of engineering and science.  Graphical models succinctly summarize modeling assumptions, and their posterior distributions naturally quantify uncertainty and produce optimal point estimators to minimize downstream risk.  Efficient, exact posterior inference algorithms exist for special cases such as the sum-product algorithm \cite{pearl1982reverend} for discrete random variables in models with low tree-width \cite{zhao2015relationship} or particular Gaussian graphical models \cite{kalman1960new}.  Outside of these special cases, however, alternative techniques must be used to approximate the posterior.  Markov chain Monte Carlo (MCMC), especially Gibbs sampling, can sample from the posterior asymptotically \cite{geman1984stochastic,hastings1970monte} but such techniques are difficult to scale to models with many observations, can have issues with mixing \cite{mossel2005phylogenetic}, and their convergence can be hard to assess \cite{vehtari2019rank}.

Variational inference (VI) \cite{jordan1999introduction} avoids sampling and instead fits an approximate posterior via optimizing a loss function that acts as a proxy for some measure of divergence between the approximate and true posteriors.  Usually this divergence is the ``reverse'' Kullback-Leibler (KL) divergence between the approximate posterior $Q$, and the true posterior $P$: $\text{KL}(Q || P)$.  Minimizing this reverse KL is equivalent to maximizing the so-called \textbf{e}vidence \textbf{l}ower \textbf{bo}und (ELBo) which depends only on the prior, the likelihood, and the approximate posterior.  Importantly, the unknown true posterior does not appear in the ELBo.  In many models, by choosing a particular space of approximate posteriors over which to optimize (the ``variational family''), maximizing the ELBo is a tractable albeit non-convex optimization problem.  In practice, such optimization problems are solved using coordinate ascent, gradient ascent, or natural gradient ascent \cite{amari1998natural}.  These optimization methods and stochastic variants thereof can scale to massive models and datasets \cite{hoffman2013stochastic}.

This formulation introduces a key tension in VI.  On one hand, simpler variational families can be more computationally tractable to optimize. On the other hand, a large approximation error will be incurred if the true posterior does not lie within the variational family.  VI is particularly tractable for fully conjugate models with all distributions belonging to exponential families, with a variational family that assumes all variables are independent.  This approach, called \emph{mean field VI}, has a number of desirable properties including computationally efficient, often analytic coordinate-wise updates that are equivalent to natural gradient steps \cite{blei2017variational}.  Many useful models--including mixture models--are not fully conjugate, but can be made conjugate by introducing additional auxiliary variables.  Unfortunately, introducing more variables and enforcing their independence in the variational family results in larger approximation gaps \cite{teh2007collapsed}.

Many approaches have been developed to extend VI beyond the mean field approximation.  Specific models, such as coupled hidden Markov models, are amenable to more expressive variational families because subsets of the variables form tractable models -- an approach called structured VI \cite{saul1996exploiting}.   For non-conjugate models or models with variational families that are otherwise not amenable to mean field VI, it may be possible to compute gradients numerically \cite{knowles2011non} or obtain an unbiased noisy estimate of the gradient via the log gradient trick (sometimes called the REINFORCE estimator) \cite{rezende2014stochastic} or in certain special cases the reparameterization trick \cite{kingma2014auto}.  Because such approaches are broadly applicable they have come to be part of the toolkit of ``black box VI'' \cite{ranganath2014black}, which seeks to automatically derive gradient estimators for very general models and variational families.  Stochastic gradient-based optimization has also been combined with structured VI \cite{hoffman2015structured}.  Additionally, approximations to gradients tailored to specific distributions have been developed \cite{jankowiak2018pathwise}.  Much recent work has been devoted to combining deep learning with Bayesian graphical models, whether purely for speeding up inference, or by using neural networks as flexible models \cite{bittner2019approximating,johnson2016composing, zhang2018advances}.  Overall, these general approaches are impressive in scope, but the gradient estimators can be noisy and hence optimizing the ELBo can require small step sizes and corresponding long compute times \cite{roeder2017sticking}.

Here, I take an alternative approach, expanding the utility of mean field VI by showing that more flexible variational families can be constructed by combining component exponential families while maintaining desirable conjugacy properties and still forming an exponential family.  This construction is particularly useful in models involving sparsity.  The approach I present maintains the analytic simplicity of mean field VI but avoids the need to introduce auxiliary variables to obtain conjugacy.

I begin with a motivating example in Section~\ref{sec:motivation}, develop the main theoretical results in Section~\ref{sec:theory}, and show their utility on two examples in Section~\ref{sec:numerical}.  Implementation details, proofs, and additional theoretical results are presented in the Appendix.

\section{A Motivating Example}
\label{sec:motivation}
To better understand the genetics of disease and other complex traits, it is now routine to collect genetic information and measure a trait of interest in many individuals.  Initially, such studies hoped to find a small number of genomic locations associated with the trait to learn about the underlying biology or to find suitable drug targets.  As more of these genome-wide association studies (GWAS) were performed, however, it became increasingly clear that enormous numbers of locations throughout the genome contribute to most traits, including many diseases \cite{boyle2017expanded}.  While this makes it difficult to prioritize a small number of candidate genes for biological followup, one can still use GWAS data to predict a trait from an individual's genotypes, an approach referred to as a polygenic score (PGS) \cite{international2009common}.  For many diseases, PGSs are sufficiently accurate to be clinically relevant for stratifying patients by disease risk \cite{khera2018genome}.

Typically PGSs are constructed by estimating effect sizes for each position in the genome and assuming that positions contribute to the trait additively.  This assumption is justified by evolutionary arguments \cite{sella2019thinking} and practical concerns about overfitting due to the small number of individuals in GWAS (tens to hundreds of thousands) relative to the number of genotyped positions (millions).  More concretely, the predicted value of the trait $\hat{Y_i}$ in individual $i$ is
\[
\hat{Y_i} := \sum_{j=1}^P G_{ij} \beta_j
\]
where $G_{ij}$ is the genotype of individual $i$ at position $j$ and $\beta_j$ is the effect size at position $j$.

Estimating the effect sizes, $\beta_j$, is complicated by the fact that to protect participant privacy GWAS typically only release marginal effect size estimates $\hat{\beta_j}$ obtained by separately regressing the value of the trait against the genotypes at each position.  These marginal estimates thus ignore the correlation in individuals' genotypes at different positions in the genome. Fortunately, these correlations are stable across sets of individuals from the same population and can be estimated from publicly available genotype data (e.g., \cite{thousand2015global}) even if that data does not contain measurements of the trait of interest.

An early approach to dealing with genotypic correlations was LDpred \cite{vilhjalmsson2015modeling} which uses MCMC to fit the following model with a spike-and-slab prior on the effect sizes:
\begin{align*}
\beta_j &\overset{\text{i.i.d.}}{\sim} p_0 \delta_0 + (1-p_0) \mathcal{N}(0, \sigma_g^2)\\
\hat{\beta} | \beta & \overset{\hphantom{\text{i.i.d.}}}{\sim} \mathcal{N}(\mathbf{X}\beta, \sigma^2_e \mathbf{X}),
\end{align*}
where $\mathbf{X} \in \mathbb{R}^{P\times P}$ is the matrix of correlations between genotypes at different positions in the genome, $\sigma^2_e$ is the variance of the estimated effect sizes from the GWAS and $\sigma^2_g$ is the prior variance of the non-zero effect sizes.

While this model was motivated by PGSs, it is also a special case of Bayesian generalized least squares with a spike-and-slab prior on the regression coefficients.

Note that if $\mathbf{X}$ is diagonal, then the likelihood is separable and the true posterior can, in fact, be determined analytically.  Unfortunately, in real data genotypes are highly correlated (tightly linked) across positions and so $\mathbf{X}$ is far from diagonal and as a result integrating over the mixture components for each $\beta_j$ must be done simultaneously, resulting in a runtime that is exponential in the number of genotyped positions.  To obtain a computationally tractable approximation to the true posterior, LDpred and a number of extensions that incorporate more flexible priors (e.g., \cite{ge2019polygenic,lloyd2019improved}) use MCMC.  As noted above, however, MCMC has a number of drawbacks and so there may be advantages to deriving and implementing a VI scheme.

A natural mean field VI approach to approximating the posterior would be to split the mixture distribution using an auxiliary random variable as follows
\begin{align}
Z_j &\overset{\text{i.i.d.}}{\sim} \text{Bernoulli}(1 - p_0) \nonumber \\
\beta_j | Z_j &\overset{\hphantom{\text{i.i.d.}}}{\sim} \mathcal{N}(0, \sigma^2_{Z_j}) \label{eq:ldpred_naive}\\
\hat{\beta} | \beta & \overset{\hphantom{\text{i.i.d.}}}{\sim} \mathcal{N}(\mathbf{X}\beta, \sigma^2_e \mathbf{X}),\nonumber 
\end{align}
with $\sigma^2_0 = 0$ and $\sigma^2_1 = \sigma^2_g$ (treating a Gaussian with zero variance as a point mass for notational convenience).  One can then use categorical variational distributions $q_{Z_j}$ for the $Z$s and Gaussian variational distributions $q_{\beta_j}$ for the $\beta$s.  Unfortunately, this approach immediately encounters an issue: when calculating the ELBo, $\text{KL}(q_{\beta_j} || \mathcal{N}(0, 0))$ is undefined unless $q_{\beta_j}$ is a point mass at zero because otherwise $q_{\beta_j}$ is not absolutely continuous with respect to the point mass at zero.  This may seem to be merely a technical issue that vanishes if $\sigma^2_0$ is taken to be some tiny value so that $q_{\beta_j}$ is absolutely continuous with respect to $\mathcal{N}(0, \sigma^2_0)$ and hence has well-defined KL, while for practical purposes $\mathcal{N}(0, \sigma^2_0)$ acts like a point mass at zero.  Yet, this superficial fix is not enough: the mean field assumption requires $Z_j$ and $\beta_j$ to be independent under the variational approximation to the posterior, which cannot capture the phenomenon that $Z_j = 0$ forces $\beta_j$ to be close to zero while $Z_j =1$ allows $\beta_j$ to be far from zero.  Further intuition is presented in Appendix~\ref{sec:analysis}, where I analyze the case with only a single position (i.e., $P=1$) in detail.

Fortunately, this problem can be solved by noting that spike-and-slab distributions like $p_0 \delta_0 + (1-p_0) \mathcal{N}(0, \sigma_g^2)$ surprisingly still form an exponential family and are conjugate to the likelihoo.  In this example, by using a spike-and-slab distribution for the variational distributions $q_{\beta_j}$, there is no need to use auxiliary variables to obtain analytical updates for fitting the approximate posterior.  A similar approach has been considered for a specific model in a different context \cite{carbonetto2012scalable}, but in the next section, I show why the analysis works for this approach and also show that it is more broadly applicable.  In Section~\ref{sec:numerical}, I return to this motivating example to show that the naive mean field VI does indeed break down and that by using sparse conjugate priors an accurate approximation to the posterior may be obtained.

\section{More flexible exponential families}
\label{sec:theory}

In this section I begin with a simple observation.  While it is usually true that mixtures of distributions from an exponential family no longer form an exponential family, that is not the case for mixtures of distributions with distinct support.  Mixtures of exponential family members with non-overlapping support always form an exponential family.

\begin{theorem}
\label{theorem:exp_family}
Let $F_1, \ldots, F_K$ be exponential families of distributions, where the distributions within each family have supports $\mathcal{S}_1,\ldots,\mathcal{S}_K$ such that $\mathcal{S}_i \cap \mathcal{S}_j = \varnothing$ for all $i \ne j$.  Further, let $\eta_1, \ldots, \eta_K$ be the natural parameters of the exponential families, $T_1, \ldots, T_K$ be the sufficient statistics, $A_1,\ldots,A_K$ be the log-partitions, and $H_1,\ldots,H_K$ be the base measures.  Then the family of mixture distributions
\begin{align*}
F_{\text{mix}} = \left\{\sum_{k=1}^K \pi_k f_k :  \sum_{k=1}^K \pi_k = 1, \pi_i \ge 0,  f_i \in F_i, \forall i \right\}
\end{align*}
is an exponential family with natural parameters
\begin{align*}
\eta_\text{mix} &= \big(\eta_1, \ldots, \eta_K, \log\pi_1 - A_1(\eta_1) - \log\pi_K + A_K(\eta_K), \ldots,\\
&\hspace{3cm} \log\pi_{K-1} - A_{K-1}(\eta_{k-1}) - \log\pi_K + A_K(\eta_K)\big),
\end{align*}
corresponding sufficient statistics
\[
T_\text{mix}(x) = \left(\mathbb{I}\left\{x\in\mathcal{S}_1\right\}T_1(x),\ldots,\mathbb{I}\left\{x\in\mathcal{S}_K\right\}T_K(x),\mathbb{I}\left\{x\in\mathcal{S}_1\right\},\ldots, \mathbb{I}\left\{x\in\mathcal{S}_{K-1}\right\}\right),
\]
log-partition
\[
A_\text{mix}(\eta_\text{mix}) = A_K(\eta_K) - \log \pi_K,
\]
and base-measure
\[
\frac{dH_\text{mix}}{dH}(x) = \prod_{i=1}^K \left(\frac{dH_i}{dH}(x)\right)^{\mathbb{I}\left\{x \in \mathcal{S}_i\right\}}
\]
where $H$ is a measure such that $H_i$ is absolutely continuous with respect to $H$ for all $i$ and $dH_i/dH$ is the usual Radon-Nikodym derivative and $0^0$ is taken to be $1$.
\end{theorem}

Note that Theorem~\ref{theorem:exp_family} has an apparent asymmetry with respect to component $K$.  This arises because the mixture weights $\pi_1,\ldots,\pi_K$ are constrained to sum to one, and hence $\pi_K$ is completely determined by the other mixture weights.  It would be possible to have a ``symmetric'' version of Theorem~\ref{theorem:exp_family} but the constraint on the mixture weights would mean that the natural parameters would live on a strictly lower dimensional space.  Such exponential families are called \emph{curved exponential families}, and many of the desirable properties of exponential families do not hold for curved exponential families.

While the restriction in Theorem~\ref{theorem:exp_family} to exponential families with non-overlapping support may seem particularly limiting, it is important to note that a new exponential family can be formed by restricting all of the distributions within any exponential family to lie within a fixed subset of their usual domain and renormalizing.  In particular, one could divide the original full domain into non-overlapping subsets and form mixtures of exponential families restricted to these subsets.   Another important set of exponential families are uniform distributions with fixed support.  While each of these exponential families only contains a single distribution, by combining mixtures of these uniform distributions with non-overlapping support, Theorem~\ref{theorem:exp_family} shows that piece-wise constant densities with fixed break points form an exponential family.  In this paper I focus primarily on the case of mixtures of a continuous exponential family member with one or more point masses.  By Theorem~\ref{theorem:exp_family}, mixtures of point masses at fixed locations form an exponential family, and diffuse distributions can be trivially restricted to falling outside of this set of measure zero, which makes clear that spike-and-slab distributions where the slab is a non-atomic distribution from an exponential family always form an exponential family.
 
While these substantially more flexible exponential families may be of independent interest, they are of little use in mean field VI unless they are conjugate to widely used likelihood models.  The following theorem provides a recipe to take a conjugate exponential family model and create a model with a more flexible prior using Theorem~\ref{theorem:exp_family} while maintaining conjugacy.
 
 \begin{theorem}
 \label{theorem:conjugacy}
 Let $F_\text{prior}(x)$ be an exponential family of prior distributions with sufficient statistics $T_\text{prior}(x)$ with the distributions supported on $\mathcal{S}_\text{prior}$ that are conjugate to an exponential family of distributions $F_{Y | X}(y | x)$.  For exponential families $F_1,\ldots,F_K$ with supports $\mathcal{S}_1,\ldots,\mathcal{S}_K$ such that $\mathcal{S}_i \cap \mathcal{S}_j = \varnothing$ for all $i\ne j$, if there exists a matrix $M_i$ and vector $v_i$ such that $T_\text{prior}(x) = M_iT_i(x) + v_i$ for all $x \in\mathcal{S}_i$ for all $i$, then $F_\text{mix}$ as defined in Theorem~\ref{theorem:exp_family} is also conjugate to $F_{Y | X}$.
 \end{theorem}
 
 Intuitively, Theorem~\ref{theorem:conjugacy} says that for a given conjugate prior distribution, say $\mathcal{P}$, we can create a more flexible prior that maintains conjugacy by combining non-overlapping component distributions so long as those component distributions have the same sufficient statistics as $\mathcal{P}$ on their domain up to an affine transformation.

Note that the point mass distribution at a fixed point is a member of an exponential family with a sole member, and hence it has no sufficient statistics.  Furthermore, any function is constant on the support of a point mass. This provides an immediate corollary that mixtures of any continuous distribution with a finite number of point-masses are conjugate to the same distributions as the original continuous distribution.  Another example is mixtures of degenerate distributions where only a subset of parameters are fixed.  For instance, consider the mean vector, $\mu$, of a multivariate Gaussian.  One could put a prior on $\mu$ that is a mixture of the usual conjugate multivariate Gaussian along with degenerate distributions like a multivariate Gaussian with the condition that $\mu_i = 0$, that $\mu_i = \mu_j$ for dimensions $i$ and $j$.  In fact, any degenerate multivariate Gaussian distribution defined by affine constraints $\mathbf{M}\mu = \mathbf{v}$ for some matrix $\mathbf{M}$ and vector $\mathbf{v}$ could be included in this mixture.  

Theorem~\ref{theorem:conjugacy} makes it extremely easy to add sparsity or partial sparsity to any conjugate model.  A non-degenerate application is constructing asymmetric priors from symmetric ones.  For example, we can construct an asymmetric prior by taking a mixture of a copy of the original prior restricted to the negative reals and a copy restricted to non-negative reals.  These mixture form an exponential family because the copies have the same sufficient statistics as the original prior but have non-overlapping domains.

\section{Numerical results}
\label{sec:numerical}

To show the applicability of the theoretical results presented in Section~\ref{sec:theory}, I test sparse VI schemes using Theorems~\ref{theorem:exp_family}~and~\ref{theorem:conjugacy}--which I refer to as the non-overlapping mixtures trick--on two models and compare these schemes to non-sparse and naive VI approximations showing the superior performance of treating sparsity exactly.  The first model is the polygenic score prediction model discussed in Section~\ref{sec:motivation} and the second is a spike-and-slab prior for sparse probabilistic principal component analysis (PCA) \cite{guan2009sparse}.  Python implementations of the fitting procedures and the results of the simulations presented here are available at \url{https://github.com/jeffspence/non_overlapping_mixtures}.  All mean field results were obtained on a Late 2013 MacBook Pro and fitting the mean field VI schemes took less than five seconds per dataset for the polygenic score model and less than two minutes per dataset for sparse PCA.

\subsection{Polygenic Scores}
\label{sec:ldpred}
While the LDpred model was originally fit using Gibbs sampling \cite{vilhjalmsson2015modeling}, it may be desirable to fit the model using VI for computational reasons.  I simulated data under the LDpred model to compare two VI schemes.  The first VI scheme is obtained by using the non-overlapping mixtures trick.  The second is the naive VI scheme of introducing auxiliary variables to split the spike-and-slab prior into a mixture, and then approximating that mixture as a mixture of Gaussians.  The details of the VI schemes are presented in Appendix~\ref{sec:vi_ldpred}.  I simulated a $1000$ dimensional vector $\hat{\beta}$ from the LDpred model with $p_0 = 0.99$ so that on average about $10$ sites had non-zero effects, a level of sparsity generally consistent with realistic human data \cite{shi2019localizing}.  For each simulation I drew $\mathbf{X}$ by simulating from a Wishart distribution with an identity scale matrix and $1000$ degrees of freedom and then dividing the draw from the Wishart distribution by $1000$.  I set $\sigma^2_1$ to be $1$ and then varied $\sigma^2_e$ from $0.05$ to $1.0$.  For each of value of $\sigma^2_e$, I simulated 100 replicate datasets and tested the VI schemes as well as a baseline that is used in statistical genetics of using the raw values of $\hat{\beta}$ as estimates of $\beta$.  I also tested adaptive random walk MCMC as implemented in NIMBLE \cite{de2017programming} which was run for 1000 iterations, and boosting black box VI (BBBVI) \cite{locatello2018boosting} as implemented in \texttt{pyro} \cite{bingham2019pyro}, which sequentially fits a mixture distribution as an approximate posterior.  For BBBVI, I used the equivalent formulation $\hat{\beta}|\beta,Z \sim \mathcal{N}(\mathbf{X}(\beta * Z), \sigma^2\mathbf{X})$, where the $\beta_i$ are independent Gaussians, and the $Z_i$ are independent Bernoullis, with $*$ being the component-wise product.  For the component distributions of the variational family, I used independent Gaussians for the $\beta_i$ and independent Bernoullis for the $Z_i$.  I used 2 particles to stochastically estimate gradients, 50 components in the mixture distribution, and 2000 gradient steps per mixture component.  The BBBVI objective functions were optimized using the \texttt{adam} optimizer \cite{kingma2015adam} with learning rate $10^{-3}$ and default parameters otherwise. I evaluated the methods using the mean squared error (MSE) and correlation between the estimated and true values of $\beta$, using the variational posterior mean as the estimator for the Bayesian approaches (Figure~\ref{fig:ldpred_main}).

\begin{figure}
\centering
\begin{subfigure}[b]{0.385\textwidth}
\centering
\includegraphics[width=\textwidth]{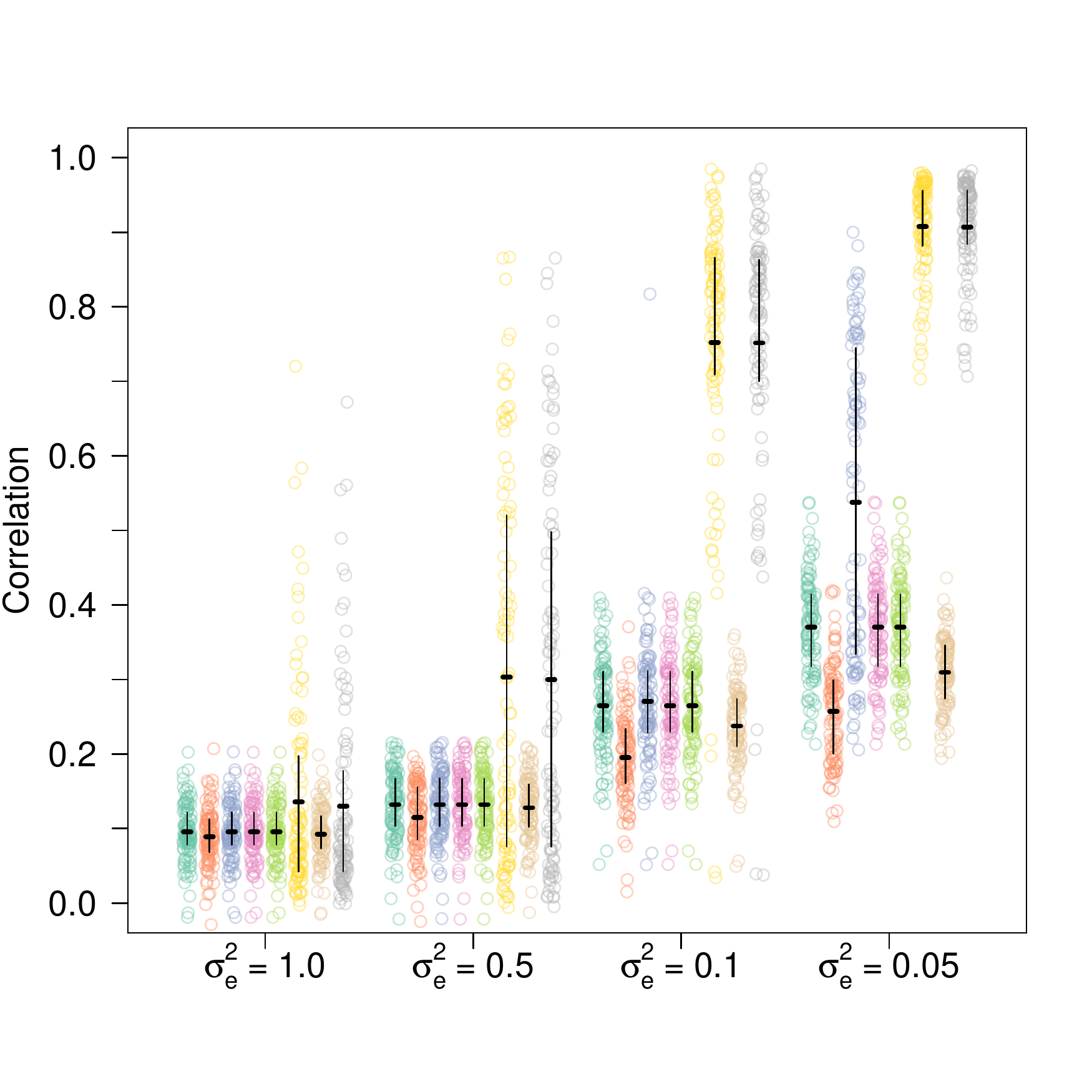}
\caption{}
\end{subfigure}
\begin{subfigure}[b]{0.495\textwidth}
\centering
\includegraphics[width=\textwidth]{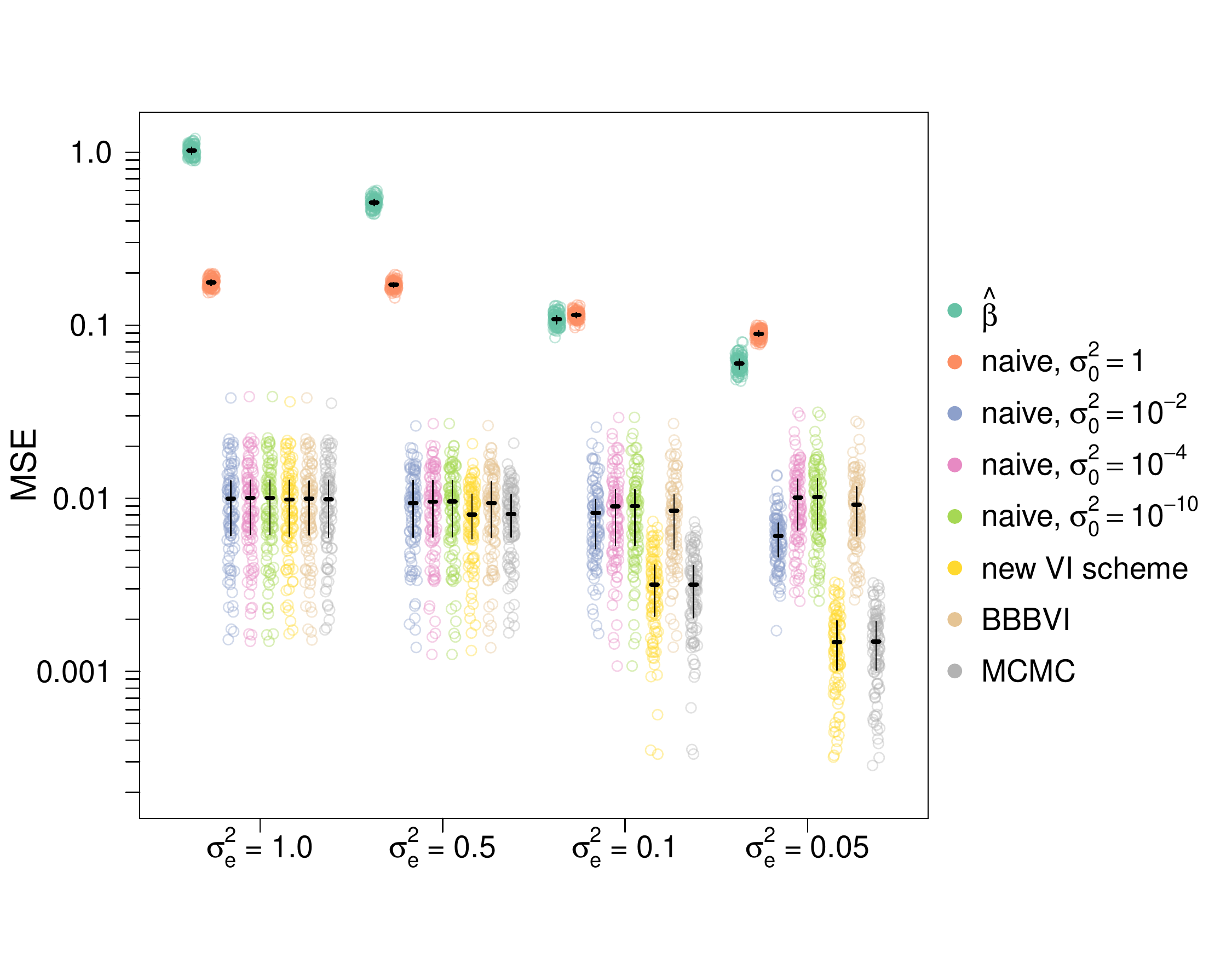}
\caption{}
\end{subfigure}
\caption{\textbf{Comparison of VI schemes for the LDpred model}\\
The non-overlapping mixtures trick (new VI scheme) is compared to the naive scheme of introducing an auxiliary variable and approximating the spike-and-slab prior by a mixture of two gaussians centered at zero, the less dispersed of which has variance $\sigma_0^2$.  As a baseline the VI schemes are compared against what is sometimes used in practice -- the raw observations, $\hat{\beta}$ as well as boosting black box VI \cite{locatello2018boosting} and adaptive random walk MCMC.  The parameter $\sigma_e^2$ controls the amount of noise, so $\sigma_e^2 = 0.05$ corresponds to a $20$ times higher signal-to-noise ration than $\sigma_e^2=1.0$.  Plot (a) compares the correlation between the estimates (posterior mean for the VI schemes, $\hat{\beta}$ for the baseline) and the true simulated values of $\beta$.  Plot (b) compares the mean squared error (MSE).  Point clouds are individual simulations, horizontal lines are means across simulations, and whiskers are interquartile ranges.  See the main text for simulation details.}
\label{fig:ldpred_main}
\end{figure}

For the naive VI scheme, there is an additional parameter $\sigma_0^2$.  When $\sigma_0^2=1=\sigma_1^2$, the method is equivalent to performing mean field VI on the non-sparse model where the prior on $\beta$ is simply a single Gaussian, which is a Bayesian version of ridge regression.  In addition to $\sigma_0^2 = 1$, I used $\sigma_0^2 \in \left\{10^{-2}, 10^{-4}, 10^{-10}\right\}$.  The results for $\sigma_0^2 = 10^{-4}$ and $\sigma_0^2 = 10^{-10}$ are indistinguishable.  All of the VI schemes (both the new scheme and the naive scheme with any value of $\sigma_0^2$) outperform the baseline of just using $\hat{\beta}$ except for in the extremely high signal-to-noise regime, where the naive model with $\sigma_0^2=1$ over-shrinks and cannot take advantage of sparsity.  BBBVI performed comparably to the naive VI schemes, but BBBVI required hours to run compared to seconds for the mean field schemes.  Meanwhile, by these metrics the non-overlapping mixtures trick performed indistinguishably from MCMC, but again took seconds per run compared to approximately 12 hours per run for MCMC. I also considered the maximum likelihood estimator (MLE) $\beta_\text{MLE} = \mathbf{X}^{-1}\hat{\beta}$ but $\mathbf{X}$ is terribly ill-conditioned resulting in high variance.  Using the MLE almost always resulted in correlations around zero (maximum correlation across all simulations was $0.13$, mean correlation was $0.005$) and extremely large MSE (minimum MSE across all simulations was 4.19, mean MSE was $\approx335000$ -- about six orders of magnitude higher than any other method).  None of the naive schemes provide a substantial improvement over the baseline in terms of correlation. In terms of MSE, the naive schemes provide some improvement if $\sigma_0^2$ is tuned properly, but the non-overlapping mixture trick outperforms all of the schemes across signal-to-noise regimes.  Exactly accounting for sparsity when the true signal is sparse substantially improves performance.

\subsection{Sparse Probabilistic PCA}
PCA \cite{pearson1901lines,hotelling1933analysis} is a widely-used exploratory data analysis method that projects high dimensional data into a low dimensional subspace define by orthogonal axes that explain a maximal amount of empirical variance.  These axes are defined by ``loadings''--weightings of the dimensions that compose a datapoint.  Dimensions with high loadings in the first few PCs are deemed ``important''  for differentiating the data points.  Unfortunately, the loadings are dense, making them difficult to interpret especially for high dimensional data.

To aid interpretability, and to leverage that in many datasets only a few variables are expected to contribute meaningfully to variation between data points, formulations of sparse PCA were developed to encourage sparsity in the loadings, usually by means of $\ell_1$ regularization \cite{zou2006sparse}.

In a parallel line of work, a Bayesian interpretation of PCA, probabilistic PCA, was developed by showing that classical PCA can be derived as the limiting maximum \emph{a posteriori} estimate of a particular generative model up to possible scaling and rotation \cite{tipping1999probabilistic}.  The probabilistic formulation more naturally extends to non-Gaussian noise models \cite{chiquet2018variational}, allows principled methods for choosing the number of principal components \cite{minka2001automatic}, gracefully handles missing data \cite{tipping1999probabilistic}, and enables speedups for structured datasets \cite{agrawal2019scalable}.

The probabilistic PCA model is
\begin{align*}
Z_1,\ldots,Z_N &\sim \mathcal{N}(0, \mathbf{I}_K)\\
X_n | Z_n &\sim \mathcal{N}(\mathbf{W}Z_n, \sigma^2_e \mathbf{I}_P)
\end{align*}
where $K$ is the number of PCs desired, $\mathbf{W} \in \mathbb{R}^{P\times K}$ is the matrix of loadings, and $Z_n$ is the PC score (i.e., projection onto the first $K$ PCs) of the $n^\text{th}$ datapoint.

These two lines of work were then brought together in sparse probabilistic PCA \cite{guan2009sparse}, which encourages sparse loadings by putting a Laplace prior on each loading, $\mathbf{W}_{pk}$.  The Laplace prior is not conjugate to the Gaussian noise model, however, but the Laplace distribution is a scale mixture of Gaussians allowing for a hierarchical decomposition, making this formulation amenable to mean field VI.

There are a number of conceptually displeasing aspects to this formulation of sparse probabilistic PCA.  First, while it is true that the maximum \emph{a posteriori} estimate of the loadings is sparse under a Laplace prior, the generative model is not sparse: because the Laplace distribution is diffuse, the loadings are non-zero almost surely.  Second, the LDpred model discussed in Section~\ref{sec:motivation} is a discrete scale mixture of Gaussians with two components.  I showed numerically in Section~\ref{sec:ldpred} and theoretically in Appendix~\ref{sec:analysis} that mean field VI breaks down in such a setting suggesting that performing mean-field VI on a scale mixture of Gaussians may be problematic especially in sparsity-inducing cases.

Motivated by these considerations I consider a spike-and-slab prior on the loadings:
\begin{align*}
\mathbf{W}_{pk} &\sim p_0\delta_0 + (1-p_0)\mathcal{N}(0, \sigma^2_1)\\
Z_n &\sim  \mathcal{N}(0, \mathbf{I}_K)\\
X_n | Z_n, \mathbf{W} &\sim \mathcal{N}(\mathbf{W}Z_n, \sigma^2_e\mathbf{I}_P)
\end{align*}
Note that \cite{guan2009sparse} considered a fully Bayesian model, which here would correspond to putting uninformative conjugate priors on $p_0$, $\sigma^2_e$, and $\sigma^2_1$.  For ease of exposition, I consider those to be fixed hyperparameters, but future work could explore putting priors on them or fitting them by maximizing the ELBo with respect to the hyperparameters in an empirical Bayes-like procedure \cite{blei2003latent}, which has been shown to automatically determine appropriate levels of sparsity in other settings \cite{wang2018empirical}.

To fit this model, I used the mean-field VI schemes described in Appendix~\ref{sec:vi_sppca}.  Briefly, I compared the performance of the naive scheme of introducing auxiliary variables, $Y_{pk} \sim \text{Bernoulli}(1-p_0)$, to split the prior on $(\mathbf{W})_{pk}$ as $(\mathbf{W})_{pk}|Y_{pk} \sim \mathcal{N}(0, \sigma^2_{Y_{pk}})$ to the scheme where the prior on $(\mathbf{W})_{pk}$ is treated exactly using the non-overlapping mixtures trick.  I compared the variational posterior mean estimates of the loadings and scores from both of these schemes to classical PCA based on SVD, as well as an ``oracle'' version of classical PCA that uses only those variables that are simulated to have non-zero loadings.

For each dataset, I simulated $500$ points with $10000$ dimensions.  The data was split into four clusters of sizes $200$, $200$, $50$, and $50$.  For $100$ of the dimensions, the value of the entry was drawn from $\mathcal{N}(\mu_c, 1)$, where $c$ indexes the cluster and each $\mu_c$ was drawn from a standard normal independently for each dimension. The remaining dimensions were drawn from standard normals. The entire matrix was then centered and scaled so that each variable had empirical mean zero, and unit empirical standard deviation, causing the simulations to differ from the generative model.  For inference I set $\sigma^2_1=0.5$, $\sigma^2_e=1$, and $p_0 = 1 - 100/10000$.  For all runs, I used $K=2$ to project onto a two-dimensional space to facilitate visualization.

\begin{figure}
\centering
\begin{subfigure}[b]{\textwidth}
\centering
\caption{}
\includegraphics[width=\textwidth]{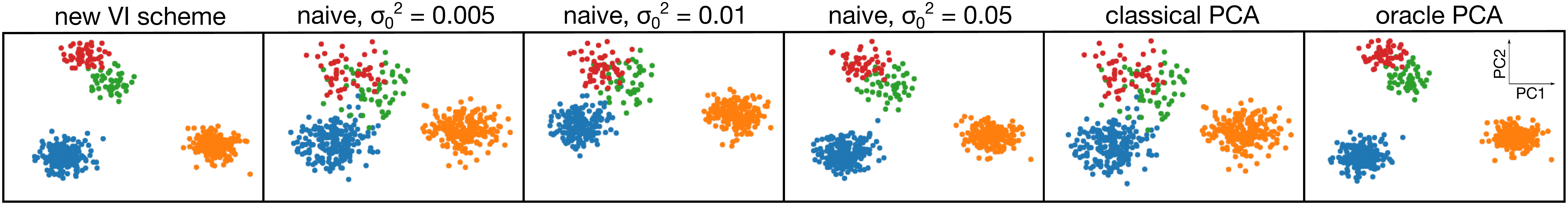}
\end{subfigure}
\begin{subfigure}[b]{0.49\textwidth}
\centering
\includegraphics[width=0.7\textwidth]{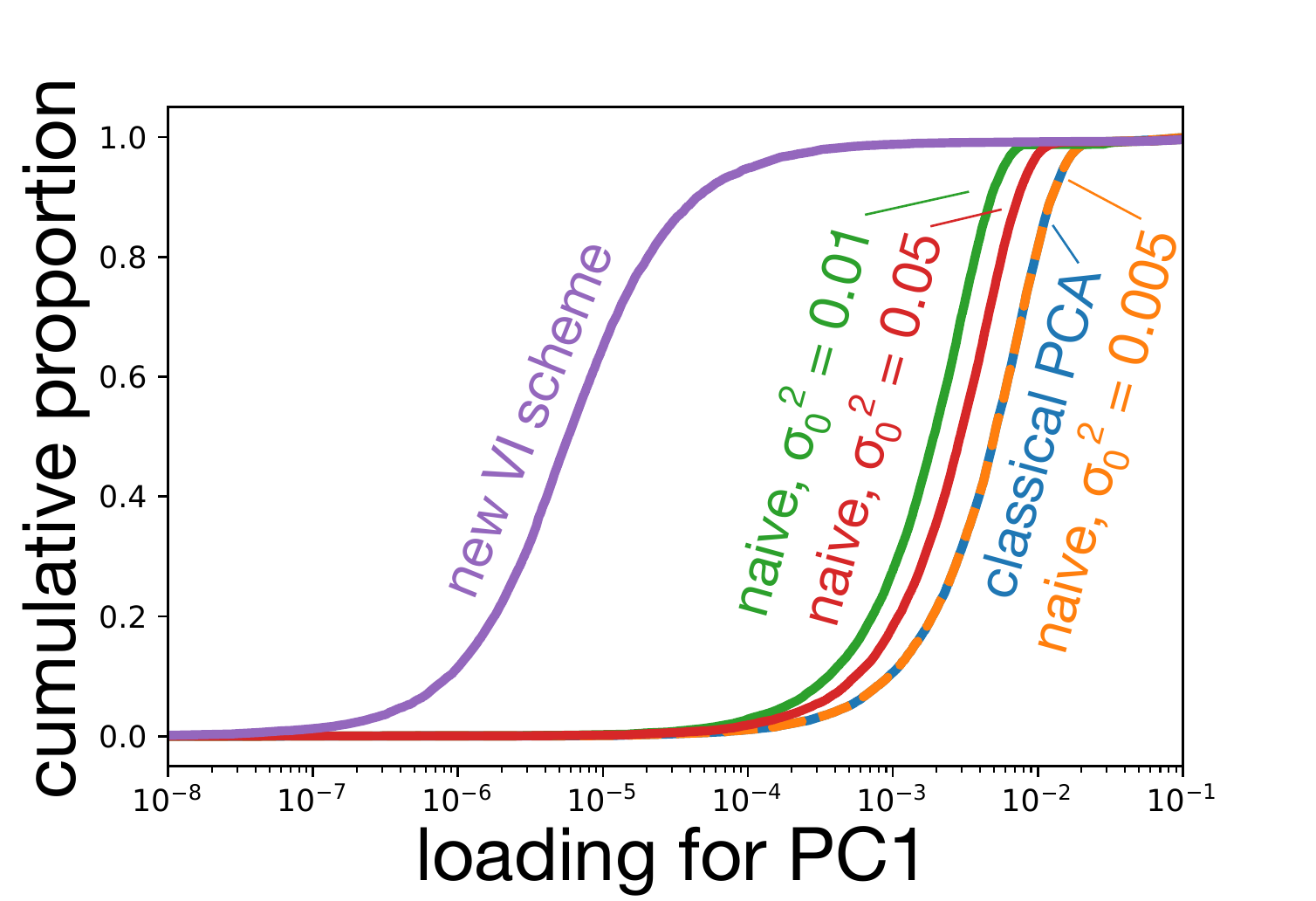}
\caption{}
\label{fig:main_loadings1}
\end{subfigure}
\begin{subfigure}[b]{0.49\textwidth}
\centering
\includegraphics[width=0.7\textwidth]{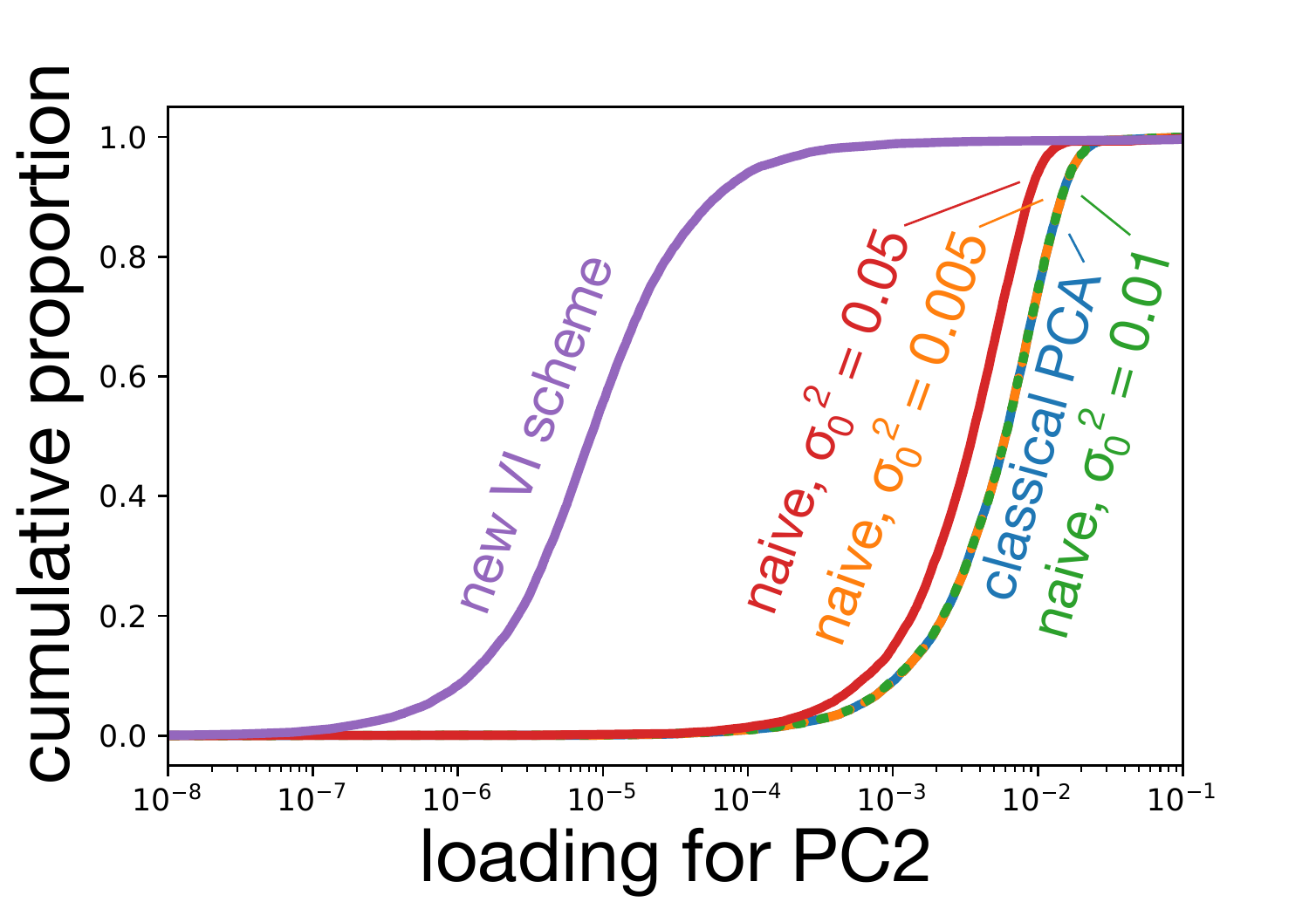}
\caption{}
\label{fig:main_loadings2}
\end{subfigure}
\caption{\textbf{Comparison of VI schemes for sparse PCA}\\
Data were simulated as described in the main text. Panel (a) shows the projections onto the first two PCs produced by various schemes. The results of the new VI scheme are visually comparable to that achievable by an ``oracle'' in this simulation scenario, while all other VI schemes do not cluster the data as well.  As $\sigma_0^2$ becomes small, the naive scheme counter-intuitively recapitulates SVD.   Panels (b) and (c) show that the loadings learned by the new VI scheme are sparse, while those of classical PCA and the various naive schemes are not.}
\label{fig:main_pca}
\end{figure}

In the naive scheme there is an additional hyperparameter, $\sigma^2_0$, for which I consider several values. Tuning this hyperparameter is crucial to obtaining reasonable results. For $\sigma^2_0 \approx \sigma^2_1$ the model is essentially probabilistic PCA with a non-sparsity inducing Gaussian prior on the loadings, while for $\sigma^2_0 \ll \sigma^2_1$ the mean-field assumption together with the zero-avoidance of VI causes the approximate posterior to put very little mass on the event $Y_{pk} = 0$, and so again the model reduces to probabilistic PCA with a Gaussian prior on the loadings.  On the other hand, the VI scheme based on the non-overlapping mixtures trick produces sensible results without requiring any tuning.  Indeed, Figure~\ref{fig:main_pca} shows that the new scheme clusters the data better than any of the naive schemes, and that as $\sigma^2_0 \downarrow 0$ the naive scheme becomes indistinguishable from classical PCA.  Furthermore, whereas the posterior mean loadings from the non-overlapping mixtures trick are indeed sparse, the loadings from the other methods are dense (Figures~\ref{fig:main_loadings1} and Figures~\ref{fig:main_loadings2}).  Additional simulations and a more quantitative measure of performance-- reconstruction error--are presented in Appendix~\ref{sec:supp_ppca}.

\section{Discussion}
VI has made it possible to scale Bayesian inference on increasingly complex models to increasingly massive datasets, but the error induced by obtaining an approximate posterior can be substantial \cite{teh2007collapsed}.  Some of this error can be mitigated by using more flexible variational families, but doing so can require alternative methods for fitting, like numerical calculation of gradients \cite{knowles2011non}, sampling-based stochastic gradient estimators \cite{kingma2014auto, rezende2014stochastic}, or other approximations \cite{jankowiak2018pathwise}.  The results of Theorems~\ref{theorem:exp_family}~and~\ref{theorem:conjugacy} provide an alternative method, using mixtures of non-overlapping exponential families to provide a more flexible variational family while maintaining conjugacy.  Even in models that are not fully conjugate, methods have been developed to exploit portions of the model that are conjugate, and the results presented here may prove useful in such schemes \cite{khan2017conjugate}.  These schemes could be especially useful in obtaining stochastic gradient updates for bayesian neural networks with spike-and-slab priors on the weights.  In particular, a ReLU applied to a Gaussian random variable is a mixture of a point mass and a Gaussian restricted to be positive, which is an exponential family by Theorem~\ref{theorem:exp_family}.

Here I focused on modeling sparse phenomena and found that the non-overlapping mixtures trick is superior to a naive approach of introducing auxiliary variables.  Yet, the pitfalls I described occur whenever mean field VI is applied to mixture distributions where the mixture components are very different.  This suggests that in some cases, it may be beneficial to use the sparse distributions presented here to approximate non-sparse mixture distributions and then treat the sparse approximation exactly.

Throughout, I assumed that the domains of the non-overlapping mixtures were specified \emph{a priori}.  This assumption could be relaxed, treating the domains as hyperparameters that could then be optimized with respect to the ELBo.  Yet, it is not obvious that for arbitrary models the objective function need to be differentiable with respect to these hyperparameters, which may necessitate the use of zeroth-order optimization procedures such as Bayesian Optimization \cite{frazier2018tutorial}.

Exponential families also play an important role in other forms of variational inference including Expectation Propagation \cite{minka2001expectation}.  The non-overlapping mixtures trick may be useful in variational approaches beyond the usual reverse KL-minimizing mean field VI.

While the non-overlapping mixtures trick makes it easy to add sparsity to conjugate models, it is not a panacea to some of the common pitfalls of VI.  For example, I also considered a sparse extension of latent Dirichlet allocation (LDA) \cite{blei2003latent, pritchard2000inference}, where documents can have exactly one topic with positive probability.  Unfortunately, the zero-avoiding nature of the reverse KL results in pathological behavior: for a document with only one topic the prior topic probabilities for each word are sparse, but in the variational posterior they must be dense.  Empirically, this results in the VI posterior only putting mass on the non-sparse mixture component and hence being indistinguishable from the usual VI approach to LDA.  In general, care should be taken when the sparsity added to a model results in the possibility of variables having zero likelihood conditioned on latent variables coming from the sparse component.

In spite of these drawbacks, providing a recipe to easily model sparsity in otherwise conjugate Bayesian models provides another avenue to model complex phenomena while maintaining the analytical and computational benefits of mean-field VI.

\section*{Broader Impact}
The primary contribution of this paper is theoretical and so the broader societal impact depends on how the theorems are used.  The polygenic score application has the possibility to improve the overall quality of healthcare, but because the majority of GWAS are performed on individuals of European ancestries, PGSs are more accurate for individuals from those ancestry groups, potentially exacerbating health disparities between individuals of different ancestries as PGSs see clinical use \cite{martin2019clinical}.  The methods presented here are equally applicable to GWAS data collected from any ancestry group, however, and so efforts to diversify genetic data will ameliorate performance differences across ancestry groups.  PGSs used for some traits such as sexual orientation \cite{ganna2019large}, educational attainment \cite{harden2020genetic}, or stigmatized psychiatric disorders \cite{international2009common} raise thorny ethical considerations, especially when the application of such PGSs could enable genetic discrimination or fuel dangerous public misconceptions about the genetic basis of such traits \cite{palk2019potential}.  On the other hand, PGSs applied to diseases have the potential to improve health outcomes and so if used responsibly could provide tremendous benefit to society.

\begin{ack}
I would like to thank Nasa Sinnott-Armstrong and Jonathan Pritchard for piquing my interest in the polygenic score application, and Jeffrey Chan, Amy Ko, Clemens Weiss, and four anonymous reviewers for helpful feedback on the manuscript. I was funded by a National Institutes of Health training grant (5T32HG000044-23).
\end{ack}

\bibliographystyle{plain}
\bibliography{sparse_exponential_family.bib}

\clearpage

\appendix
\section{Analysis of naive Mean Field VI for the LDpred Model when $P=1$}
\label{sec:analysis}
When there is only one mutation, the naive mean field VI approach to the LDpred model simplifies to
\begin{align*}
Z &\sim \text{Bernoulli}(1 - p_0)\\
\beta | Z &\sim \mathcal{N}(0, \sigma^2_{Z})\\
\hat{\beta} | \beta &\sim \mathcal{N}(\beta, \sigma^2_e).
\end{align*}
In this simplified setting it is possible to obtain a closed form expression for the posterior.  For the purposes of contrasting with VI, I consider the posterior probability that $\beta$ comes from each component of the mixture distribution, $p(Z=z | \hat{\beta})$ as well as the posterior mean of $\beta$, $\mathbb{E}[\beta | \hat{\beta}]$.  In particular, for the case $\sigma_0^2 = 0$,
\begin{align}
p(Z=0| \hat{\beta}) &= \frac{p_0}{p_0 + (1-p_0)\sqrt{2\pi\frac{1}{\frac{1}{\sigma_e^2} + \frac{1}{\sigma_1^2}}} \exp\Big\{\frac{\hat{\beta}^2}{2(\sigma_e^2 + \frac{\sigma_e^4}{\sigma_1^2})}\Big\}}\label{eq:true_post}
\\
\mathbb{E}[\beta | \hat{\beta}] &= \Big(1 - p(Z=0 | \hat{\beta})\Big) \frac{\hat{\beta}}{\frac{\sigma_e^2}{\sigma_1^2}+1}.\label{eq:true_mean}
\end{align}

In this case, the usual approach to mean field VI would be to find an approximate posterior that factorizes $q(\beta, Z) = q(\beta)q(Z)$ and assume that $q(\beta)$ and $q(Z)$ are conditionally conjugate, which in this case would be that $q(\beta) = \mathcal{N}(\mu, s^2)$ and $q(Z) = \text{Bernoulli}(1 - \psi_0)$.  As stated in the main text, the ELBo is undefined if $\sigma_0^2 =0$, so instead consider $\sigma_0^2$ to be small but nonzero.

Under these assumptions, I show that for any $\hat{\beta}$ there is a $\sigma_0^2$ small enough such that the probability under $q$ that $Z=1$ is approximately either 0 or 1 and as a result the variational posterior mean of beta is either approximately 0 or approximately equally to the non-sparse case where $p_0=0$.  That is, mean field VI either over-shrinks effects to zero or provides no more shrinkage than just having a single gaussian prior on the effect sizes.  In contrast note that $p(Z=0 |\hat{\beta})$ varies smoothly as a function of $\hat{\beta}$, and consequently by Equation~\ref{eq:true_mean}, the posterior mean varies smoothly from shrinking tiny effects to zero to providing less shrinkage for large effects.

\begin{theorem}
Let $q_{\hat{\beta}, \sigma_0^2}(\beta, Z)$ be the approximate posterior obtained from the LDpred model with $P=1$ for data $\hat{\beta}$.  For any $\delta$, there exists an $\epsilon$ such that for all $\sigma_0^2 < \epsilon$, either:
\begin{align}
q_{\hat{\beta}, \sigma_0^2}(Z = 0) &\ge 1 - \delta \label{eq:thm1}\\
\Big|\mathbb{E}_{q_{\hat{\beta}, \sigma_0^2}}[\beta]\Big| &\le \delta \label{eq:thm2}
\end{align}
or
\begin{align}
q_{\hat{\beta}, \sigma_0^2}(Z = 0) &\le \delta \label{eq:thm3}\\
\Big|\mathbb{E}_{q_{\hat{\beta}, \sigma_0^2}}[\beta]\Big| &\ge  \frac{|\hat{\beta}|}{\frac{\sigma_e^2}{\sigma_1^2} + 1} - \delta, \label{eq:thm4}
\end{align}
with the case depending on the values of $p_0$, $\sigma_e^2$, $\sigma_1^2$, and $\hat{\beta}$.
\end{theorem}
\begin{proof}
I begin by writing the ELBo:
\begin{align}
\text{ELBo} &= \mathbb{E}_{q_{\hat{\beta}, \sigma_0^2}}[\log p (\hat{\beta} | \beta)] - \text{KL}(q_{\hat{\beta}, \sigma_0^2}(\beta, Z) || p(\beta, Z)) \nonumber\\
&= \text{constant} - \frac{1}{2 \sigma_e^2}(\mu^2 + s^2 - 2\hat{\beta}\mu) + \frac{1}{2}\log s^2 - \psi_0\log\psi_0 - (1-\psi_0)\log (1 -\psi_0) \nonumber\\
&\hspace{1cm} - \frac{1}{2}\psi_0\log \sigma^2_0 -\frac{1}{2}(1-\psi_0)\log \sigma_1^2 - \Big(\frac{\psi_0}{2\sigma_0^2} + \frac{1-\psi_0}{2\sigma_1^2}\Big)(\mu^2 + s^2) \nonumber\\
&\hspace{1cm}+ \psi_0 \log p_0 + (1 - \psi_0) \log(1-p_0) \label{eq:elbo}.
\end{align}
Taking partial derivatives I arrive at the equations for the critical points for $\mu$ and $s^2$:
\begin{align*}
\frac{d\text{ELBo}}{d\mu} &= -\frac{\hat{\beta}}{\sigma^2_e} - \Big(\frac{1}{\sigma_e^2} + \frac{\psi_0}{\sigma_0^2} + \frac{1-\psi_0}{\sigma_1^2}\Big)\mu\\
\frac{d\text{ELBo}}{d s^2} &= -\frac{1}{2\sigma_e^2} + \frac{1}{2s^2} - \frac{\psi_0}{2\sigma_0^2} - \frac{1 - \psi_0}{2\sigma_1^2}.
\end{align*}
Rearranging I obtain
\begin{align}
\mu &= \frac{\hat{\beta}}{1 + \psi_0\frac{\sigma_e^2}{\sigma_0^2} + (1-\psi_0)\frac{\sigma_e^2}{\sigma_1^2}}\label{eq:mu}\\
s^2 &= \frac{1}{1/\sigma_e^2 + \psi_0 / \sigma_0^2 + (1-\psi_0) / \sigma_1^2}\label{eq:sigma}.
\end{align}
Now, I show that for $\psi_0 = 0$ or $\psi_0 = 1$ the ELBo is larger than for any other $\psi_0$ so long as $\lim_{\sigma_0^2\downarrow0} \psi_0 / \sigma_0^2 > 0$.  This indicates that the optimal value of $\psi_0$ must converge to either 0 or 1 in the limit of small $\sigma_0^2$ and furthermore, if $\psi_0$ converges to $0$ it must do so faster than $\sigma_0^2$.  Taking limits in Equations~\ref{eq:mu}~and~\ref{eq:sigma} under these conditions gives Equations~\ref{eq:thm1},~\ref{eq:thm2},~\ref{eq:thm3},~and~\ref{eq:thm4}.

If $\psi_0 = 0$, then plugging $\psi_0$ into Equations~\ref{eq:mu}~and~\ref{eq:sigma}, it is clear that the values of both $\mu$ and $s^2$ are independent of $\sigma_0^2$.  Therefore, $\text{ELBo}(\psi_0=0) = O(1)$.  

On the other hand, plugging $\psi_0 = 1$ into Equations~\ref{eq:mu}~and~\ref{eq:sigma} gives $\mu = \sigma_0^2 \hat{\beta} / (\sigma_0^2 + \sigma_e^2)$ and $s^2 = \sigma_0^2 / (1 + \frac{\sigma_0^2 }{\sigma_e^2})$.  Therefore, $\log s^2 = \log \sigma_0^2 - \log (1 + \frac{\sigma_0^2 }{\sigma_e^2}) = \log \sigma_0^2 + O(\sigma_0^2)$, and $\mu^2 + s^2 = O(\sigma_0^2)$.  Plugging these results into the ELBo of Equation~\ref{eq:elbo} gives
\[
\text{ELBo}(\psi_0=1) = \frac{1}{2}\log s^2 - \frac{1}{2}\log\sigma_0^2 - \frac{1}{2\sigma_0^2} (\mu^2 + s^2) + O(1) = O(1).
\]
Now, for fixed $\psi_0 \in (0, 1)$,
\begin{align*}
\log s^2 &= \log \sigma_0^2 - \log \Big(\frac{\sigma_0^2}{\sigma_e^2} + \psi_0 + (1-\psi_0) \frac{\sigma_0^2}{\sigma_1^2} \Big) = \log \sigma_0^2  + O(\sigma_0^2)\\
\mu^2 + s^2 &= O(\sigma_0^2)
\end{align*}
which gives an ELBo of
\[
\text{ELBo}(0 < \psi_0 < 1) = \frac{1}{2}\log s^2 -\frac{1}{2}\psi_0 \log \sigma_0^2 - \frac{\psi_0}{2\sigma_0^2} (\mu^2 + \sigma^2) + O(1) = \frac{1}{2}(1 - \psi_0) \log \sigma_0^2 + O(1).
\]
For $\psi_0 \in (0, 1)$,
\[
\lim_{\sigma_0^2 \downarrow 0} \frac{1}{2}(1 - \psi_0) \log \sigma_0^2 = -\infty
\]
showing that in the limit of small $\sigma_0^2$, $\psi_0$ must converge to either $0$ or $1$.  Now, because $\psi_0 / \sigma_0^2$ appears in Equations~\ref{eq:mu}~and~\ref{eq:sigma}, some care must be taken in the case where $\psi_0$ converges to $0$.  In particular, I show that the ELBo is larger when $\psi_0=0$ than it is when $\lim_{\sigma_0^2 \downarrow 0} \psi_0 / \sigma_0^2 = c$, for some positive, finite constant $c$ so terms like $\psi_0 / \sigma_0^2$ can be neglected in the limit when obtaining Equation~\ref{eq:thm4} from Equation~\ref{eq:mu}.

Noting that by Equations~\ref{eq:mu}~and~\ref{eq:sigma}
\[
\mu = \frac{s^2}{\sigma_e^2}\hat{\beta},
\]
and collecting terms and rearranging the ELBo assuming that $\psi_0 < 1$ and $\psi_0 \downarrow 0, \sigma_0^2 \downarrow 0, \psi_0/\sigma_0^2 \rightarrow c$ results in
\begin{align*}
\text{ELBo} &= -\frac{1}{2} + \frac{\hat{\beta}^2s^2}{2\sigma_e^4} + \frac{1}{2}\log s^2 - \psi_0\log\psi_0 - (1-\psi_0)\log(1-\psi_0)\\
&\hspace{1cm} - \frac{1}{2} \psi_0 \log\sigma_0^2 -\frac{1}{2}(1-\psi_0) \log \sigma_1^2 + \psi_0\log p_0 + (1-\psi_0)\log(1-p_0)\\
&=  -\frac{1}{2} + \log(1-p_0) + \frac{\hat{\beta}^2s^2}{2\sigma_e^4} + \frac{1}{2}\log s^2 + o(1).
\end{align*}
Now considering the ELBo as a function of $\psi_0$, I consider the difference of the ELBo evaluated at $\psi_0=0$, and that evaluated at $\psi_0 < 1$ which will be denoted as $\Delta\text{ELBo}$.
\begin{align*}
\Delta\text{ELBo} &= \frac{\hat{\beta}^2}{2\sigma_e^4}(s^2|_{\psi_0=0} - s^2|_{\psi_0}) + \frac{1}{2} \log s^2|_{\psi_0=0} - \frac{1}{2}\log s^2|_{\psi_0} + o(1)\\
&> \frac{\hat{\beta}^2}{2\sigma_e^4}(s^2|_{\psi_0=0} - s^2|_{\psi_0}) + o(1)\\
&= \frac{\hat{\beta}^2}{2\sigma_e^2}\left(\frac{1}{1/\sigma_e^2 + 1/\sigma_1^2} - \frac{1}{1/\sigma_e^2 + \psi_0/\sigma_0^2 + (1-\psi_0)/\sigma_1^2)}\right) + o(1)\\
&=  \frac{\hat{\beta}^2}{2\sigma_e^2}\left(\frac{\psi_0/\sigma_0^2 - \psi_0/\sigma_1^2}{(1/\sigma_e^2 + 1/\sigma_1^2)(1/\sigma_e^2 + \psi_0/\sigma_0^2 + (1-\psi_0)/\sigma_1^2)}\right) + o(1)\\
&=  \frac{\hat{\beta}^2}{2\sigma_e^2}\left(\frac{c}{(1/\sigma_e^2 + 1/\sigma_1^2)(1/\sigma_e^2 + c +1/\sigma_1^2)}\right) + o(1)
\end{align*}
where the first inequality follows from the fact that $s^2$ is largest when $\psi_0=0$.  This quantity is obviously positive for any $\sigma_0^2$ sufficiently small and therefore if $\psi_0$ converges to $0$, it must do so faster that $\sigma_0^2$ completing the proof.
\end{proof}

The fact that under the VI approximate posterior $q(Z=1)$ is either close to $0$ or close to $1$, while under the true posterior, $p(Z=1 | \hat{\beta})$ varies smoothly as a function of $\hat{\beta}$ suggests a thresholding phenomenon where for $\hat{\beta}$ slightly less than the threshold, the VI approximate posterior dramatically over shrinks, while for $\hat{\beta}$ slightly greater than the threshold the VI approximate posterior dramatically under shrinks essentially performing hard thresholding.  In Figure~\ref{fig:thresh} I show numerically that this is indeed the case, highlighting the failure of mean field VI to provide a reasonable approximation to the posterior for even this toy model.

\begin{figure}
\centering
\includegraphics{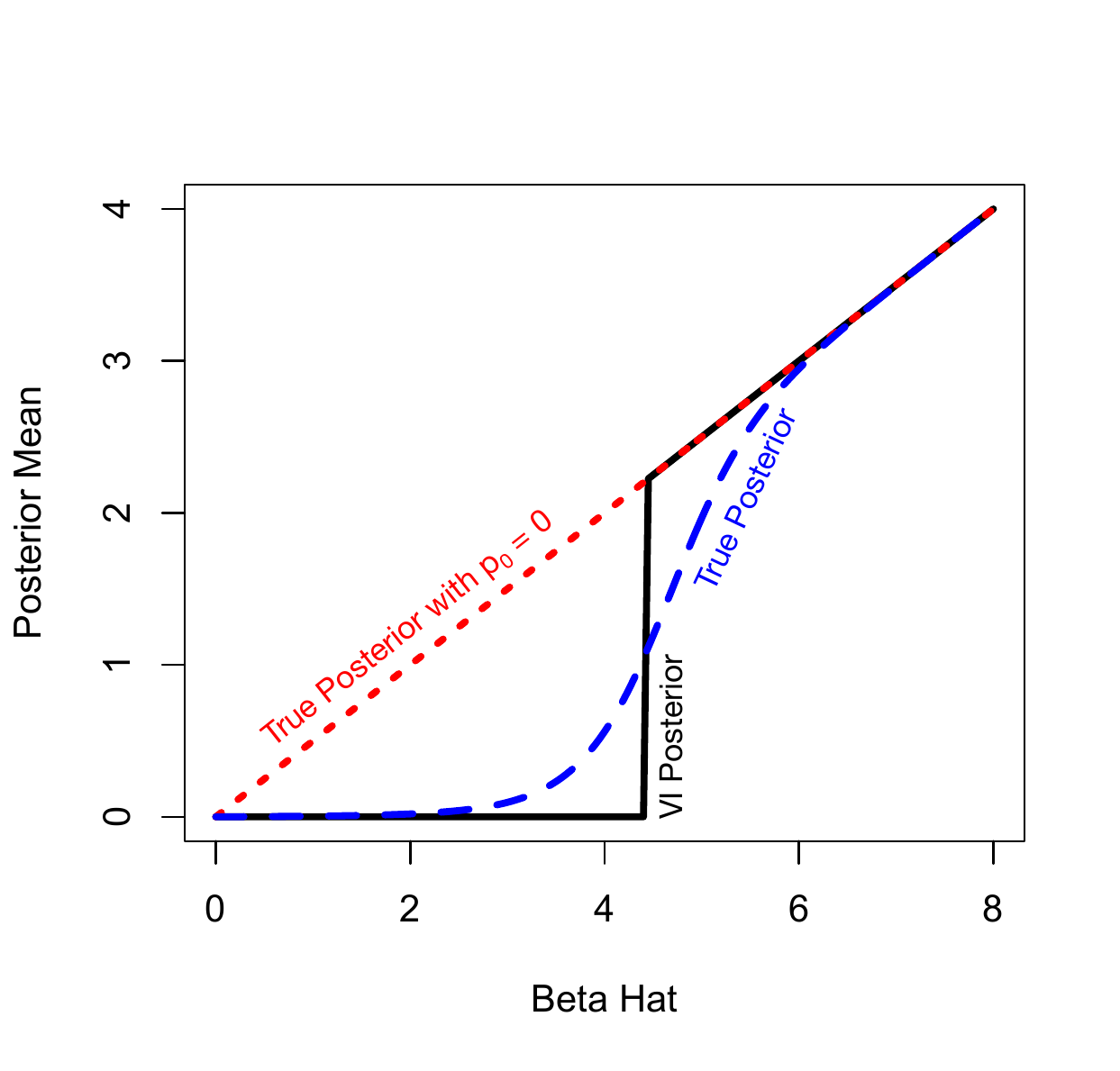}
\caption{\textbf{Thresholding phenomenon for naive mean field VI for the LDpred model}\\
The approximate posterior mean when using naive mean field VI undergoes a thresholding phenomenon.  For values of $\hat{\beta}$ close to zero, the VI scheme over-shrinks the posterior mean essentially to zero, while above the threshold, the VI scheme under-shrinks essentially matching the model without sparsity.  The results here were generated with $p_0 = 0.99, \sigma_1^2 = \sigma_e^2 = 1$ and the VI model was fit using $\sigma_0^2 \approx 10^{-22}$. The results are qualitatively similar for all $\sigma_0^2 \le 0.01$, and for larger $\sigma_0^2$ the VI model significantly under-shrinks for small $\hat{\beta}$.}
\label{fig:thresh}
\end{figure}

\section{Proof of Theorem~\ref{theorem:exp_family}}
\begin{proof}
First, note that any measure $f_\text{mix} \in F_{\text{mix}}$ may clearly be written as 
\[
\frac{df_\text{mix}}{dH}(x) = \sum_{i=1}^K \mathbb{I}\left\{x \in \mathcal{S}_i\right\} \pi_i \exp\left\{\langle\eta_i,T_i\rangle - A_i(\eta_i)\right\} \frac{dH_i}{dH}(x),
\] 
because each $f_i$ is a member of an exponential family.  Since all of the events in the indicators are mutually exclusive by hypothesis, the resulting measure may be re-written as
\begin{align*}
\frac{df_\text{mix}}{dH}(x) &= \prod_{i=1}^K \exp\left\{\mathbb{I}\left\{x \in \mathcal{S}_i\right\}\left(\log\pi_i + \langle\eta_i,T_i\rangle - A_i(\eta_i)\right)\right\} \left(\frac{dH_i}{dH}(x)\right)^{\mathbb{I}\left\{x \in \mathcal{S}_i\right\}}\\
&= \left[\prod_{i=1}^K \left(\frac{dH_i}{dH}(x)\right)^{\mathbb{I}\left\{x \in \mathcal{S}_i\right\}}\right] \exp\left\{\sum_{i=1}^K \mathbb{I}\left\{x \in \mathcal{S}_i\right\}\left(\log\pi_i + \langle\eta_i,T_i\rangle - A_i(\eta_i)\right)\right\}\\
&= \left[\prod_{i=1}^K \left(\frac{dH_i}{dH}(x)\right)^{\mathbb{I}\left\{x \in \mathcal{S}_i\right\}}\right] \exp\left\{\sum_{i=1}^K \mathbb{I}\left\{x \in \mathcal{S}_i\right\}\left(\log\pi_i - A_i(\eta_i)\right) + \langle\eta_i, \mathbb{I}\left\{x \in \mathcal{S}_i\right\}T_i\rangle\right\}\\
&= \frac{dH_\text{mix}}{dH}(x)\exp\left\{\langle\eta_\text{mix}, T_\text{mix}(x)\rangle - A_\text{mix}(\eta_\text{mix})\right\}
\end{align*}
completing the proof.
\end{proof}

\section{Proof of Theorem~\ref{theorem:conjugacy}}
 \begin{proof}
 Begin by noting that by the conjugacy conditions (see e.g., \cite{diaconis1979conjugate}), for any measure $f _{Y|X} \in F_{Y|X}$,
 \[
 df_{Y|X}(y|x) = dH^*(y) \exp \left\{ \left\langle \left[T^*_Y(y), \alpha \right], T_\text{prior}(x) \right\rangle\right\},
 \]
 for some $\alpha$ and $H^*(y)$, where $T_\text{prior}$ is assumed without loss of generality to be ordered in a particular way, and $T^*_Y$ are the subset of sufficient statistics of $f_{Y|X}$ that are coefficients of $T_\text{prior}$.
 
 Now, since $\mathbb{I}\left\{x \in \mathcal{S}_i\right\}$ are mutually exclusive and exactly one such event occurs for each $x$,
 \begin{align*}
  df_{Y|X}(y|x) &= dH^*(y) \exp \left\{ \sum_{i=1}^K\left\langle \left[T^*_Y(y), \alpha \right], \mathbb{I}\left\{x \in \mathcal{S}_i\right\}T_\text{prior}(x) \right\rangle\right\}\\
  &= dH^*(y) \exp \left\{ \sum_{i=1}^K\left\langle \left[T^*_Y(y), \alpha \right], \mathbb{I}\left\{x \in \mathcal{S}_i\right\}(M_iT_i(x) + v_i)\right\rangle\right\} \\
  &= dH^*(y) \exp \Bigg\{ \sum_{i=1}^K\left\langle M_i^T\left[T^*_Y(y), \alpha \right], \mathbb{I}\left\{x \in \mathcal{S}_i\right\}T_i(x)\right\rangle\\
  &\hspace{3cm}+\sum_{i=1}^{K-1}(v_i^T\left[T^*_Y(y), \alpha \right] - v_K^T\left[T^*_Y(y), \alpha \right]) \mathbb{I}\left\{x \in \mathcal{S}_i\right\}\\
  &\hspace{3cm}+v_K^T\left[T^*_Y(y), \alpha \right]\Bigg\}
 \end{align*} 
 where I used the hypothesis on the sufficient statistics $T_1,\ldots,T_K$ for the second equality.   Multiplying by an arbitrary measure $f_\text{mix} \in F_\text{mix}$ and collecting terms I obtain
 \begin{align*}
df_{X|Y}(x|y) &\propto   \exp \Bigg\{ \sum_{i=1}^K\left\langle M_i^T\left[T^*_Y(y), \alpha \right] + \eta_i, \mathbb{I}\left\{x \in \mathcal{S}_i\right\}T_i(x)\right\rangle\\
  &\hspace{2cm}+\sum_{i=1}^{K-1}\Big(v_i^T\left[T^*_Y(y), \alpha \right] - v_K^T\left[T^*_Y(y), \alpha \right] \\&\hspace{3.5cm}+ \log\pi_i - A_i(\eta_i) - \log\pi_K + A_K(\eta_K)\Big) \mathbb{I}\left\{x \in \mathcal{S}_i\right\}\rangle\Bigg\}dH_\text{mix}(x)\\
  &= \exp \left\{\langle \eta_\text{mix}^*, T_\text{mix}(x) \rangle\right\}dH_\text{mix}(x),
 \end{align*}
 where $\eta_\text{mix}^*$ is the updated parameter obtained by collecting terms, showing that the posterior is in the same exponential family as the prior.
\end{proof}

\section{VI Schemes for the LDpred model}
\label{sec:vi_ldpred}
Recall that the naive VI scheme introduces auxiliary variables to approximately model sparsity as in Equation~\ref{eq:ldpred_naive}.  The natural mean field approach would then be to approximate the posterior over $\beta_j$ as a Gaussian with mean $\mu_j$ and variance $s^2_j$, and approximate the posterior over $Z_j$ as a Bernoulli with probability of being zero $\psi_j$ to maintain conditional conjugacy.

Routine calculations then show that the coordinate ascent updates are
\begin{align*}
\psi_i &\leftarrow \frac{p_0 \exp \left\{-\frac{1}{2}\log\sigma_0^2 - \frac{1}{2\sigma_0^2}(\mu_i^2 + s_i^2)\right\}}{p_0 \exp \left\{-\frac{1}{2}\log\sigma_0^2 - \frac{1}{2\sigma_0^2}(\mu_i^2 + s_i^2)\right\} + (1-p_0) \exp \left\{-\frac{1}{2}\log\sigma_1^2 - \frac{1}{2\sigma_1^2}(\mu_i^2 + s_i^2)\right\}}\\
\mu_i &\leftarrow \frac{\hat{\beta}_i - \sum_{j\ne i} \mathbf{X}_{ij}\mu_j}{\psi_i \sigma_e^2/\sigma_0^2 + (1-\psi_i)\sigma_e^2/\sigma_1^2 + \mathbf{X}_{ii}}\\
s^2_i &\leftarrow \frac{1}{\psi_i / \sigma_0^2 + (1-\psi_i) / \sigma_1^2 + \mathbf{X}_{ii} / \sigma_e^2}.
\end{align*}

Using Theorems~\ref{theorem:exp_family}~and~\ref{theorem:conjugacy} it is possible to derive an alternative VI scheme that eschews the need for auxiliary variables.  In particular, the set of distributions containing only the point mass at 0 is trivially an exponential family, and the support of distributions in that family do not overlap with the set of Gaussians supported on $\mathbb{R}\setminus \left\{0\right\}$.  Therefore, the set of distributions that are mixtures of a Gaussian and a point mass at 0 are also an exponential family by Theorem 1.  Then, by Theorem 2, because a Gaussian prior on the mean of a Gaussian is conjugate, and the sufficient statistics of a Gaussian are constant on the set $\left\{0\right\}$, this mixture distribution is also a conjugate prior for the mean of a Gaussian.  The natural approximation to make for the variational posterior over $\beta_j$ would then lie in the same exponential family -- a mixture of a Gaussian with mean $\mu_j$ and variance $s^2_j$ and a point mass at $0$, with the probability of 0 being $\psi_j$.

Because the model is conjugate and the distributions are in the exponential family, the optimal updates for the natural parameters can be obtained from
\begin{align}
\label{eq:ldpredupdate}
q(\beta_i) \propto \exp \left\{ \mathbb{E}_{-i}\log P(\hat{\beta} | \beta) + \log P(\beta_i) \right\}
\end{align}
where $\mathbb{E}_{-i}[\cdot]$ is short hand for taking the expectation under the approximate posterior with respect to all variables except $\beta_i$.  The posterior mean under the variational approximation is $\mathbb{E}_q[\beta_i] = (1-\psi_i)\mu_i$, and so the first term expands to
\begin{align*}
\mathbb{E}_{-i}\log P(\hat{\beta} | \beta) &= -\frac{1}{2\sigma^2_e} \mathbb{E}_{-i}(\hat{\beta} - \mathbf{X}\beta)^T\mathbf{X}^{-1}(\hat{\beta} - \mathbf{X}\beta)\\
&= -\frac{1}{2\sigma_e^2} \mathbf{X}_{ii} \beta_i^2   + \frac{1}{\sigma_e^2}  \left(\hat{\beta}_i - \sum_{j \ne i} \mathbf{X}_{ij}(1-\psi_j)\mu_j\right) \beta_i + \text{const}\\
&=  -\frac{1}{2\sigma_e^2} \mathbf{X}_{ii}  \mathbb{I}\left\{\beta_i \ne 0\right\}\beta_i^2   + \frac{1}{\sigma_e^2}  \left(\hat{\beta}_i - \sum_{j \ne i} \mathbf{X}_{ij}(1-\psi_j),\mu_j\right) \mathbb{I}\left\{\beta_i \ne 0\right\} \beta_i + \text{const}.
\end{align*}
The natural parameters for a Gaussian with mean $\mu_i$ and variance $s^2_i$ are $-\frac{1}{2s_i^2}$ and $\frac{\mu_i}{s_i^2}$ with log normalizer $\frac{\mu_i^2}{2s_i^2} + \frac{1}{2}\log s_i^2$, with corresponding sufficient statistics $\beta_i^2$ and $\beta_i$.  By Theorem 1, the natural parameters for the mixture distribution are therefore $-\frac{1}{2s_i^2}$, $\frac{\mu_i}{s_i^2}$, and $\log \psi_i - \log (1-\psi_i) + \frac{\mu_i^2}{2s_i^2} + \frac{1}{2}\log s_i^2$, with corresponding sufficient statistics  $\mathbb{I}\left\{\beta_i \ne 0\right\} \beta_i^2$, $\mathbb{I}\left\{\beta_i \ne 0\right\} \beta_i$, and $\mathbb{I}\left\{\beta_i = 0\right\}$.  Matching the coefficients of the sufficient statistics in Equation~\ref{eq:ldpredupdate} and performing some algebra produces

\begin{align*}
\psi_i &\leftarrow 1 - \frac{1}{1 + \frac{p_0}{1-p_0}\sqrt{1 + \mathbf{X}_{ii} \sigma_1^2 / \sigma_e^2}  \exp\left\{ \frac{-\Big(\hat{\beta}_i - \sum_{j\ne i } \mathbf{X}_{ij} \mu_j (1 - \psi_j)\Big)^2}{2\sigma_e^4/ \sigma_1^2 + 2\sigma_e^2\mathbf{X}_{ii}}\right\}}\\
\mu_i &\leftarrow \frac{\hat{\beta}_i - \sum_{j\ne i } \mathbf{X}_{ij} \mu_j (1 - \psi_j)}{\sigma_e^2 / \sigma_1^2 + \mathbf{X}_{ii}}\\
s^2_i &\leftarrow \frac{1}{1/\sigma_1^2 + \mathbf{X}_{ii} / \sigma_e^2}.
\end{align*}

When fitting either VI scheme, I performed 100 iterations of coordinate ascent using the above update.  For the naive scheme, for coordinate $i$, I update $\mu_i$ and $s_i^2$ first, then update $\psi_i$ before moving on to coordinate $i+1$.  For initialization, $\mu_i = 0$ for all $i$, and $s_i^2 = \sigma_1^2 + \sigma_e^2$.  For the naive case, $\psi_i$ was initialized to be $1$, while for new scheme, $\psi_i$ was initialized to be $p_0$.

In both VI schemes, the rate-limiting step is clearly computing terms that involve summations of the type $\sum_{j\ne i}$, which take $O(P)$ time, where $P$ is the number of variables.  Since there are $O(P)$ variational parameters to update at each iteration, the runtime of each iteration is thus $O(P^2)$.

\section{VI Schemes for sparse probabilistic PCA}
\label{sec:vi_sppca}
First I derive the naive VI scheme.  For the auxiliary model,
\begin{align*}
Y_{pk} &\sim \text{Bernoulli}(1-p_0)\\
\mathbf{W}_{pk} | Y_{pk} &\sim \mathcal{N}(0, \sigma_{Y_{pk}}^2)\\
Z_n &\sim \mathcal{N}(0, \mathbf{I}_K)\\
X_n | Z_n, \mathbf{W} &\sim \mathcal{N}(\mathbf{W}Z_n, \sigma_e^2 \mathbf{I}_P)
\end{align*}
The natural mean field VI scheme for this model would be to assume that all variables are independent and assume that under the posterior $Y_{pk}$ is $\text{Bernoulli}$ with parameter $\psi_{pk}$, $\mathbf{W}_{pk}$ is Gaussian with mean $\mu_{W_{pk}}$ and variance $s^2_{W_{pk}}$, and $Z_n$ is multivariate normal with mean $\mu_{Z_n}$ and covariance matrix $\mathbf{S}_{Z_n}$. Below, I use the notation
\[
\mathbf{X} := \begin{pmatrix} \vert & & \vert \\ X_1 & \cdots & X_n\\ \vert & & \vert \end{pmatrix}.
\]
Routine calculations result in the following updates:
\begin{align*}
\mu_{Z_n} &\leftarrow \frac{1}{\sigma_e^2}\left(\frac{1}{\sigma_e^2}\mathbb{E}[\mathbf{W}^T\mathbf{W}] + \mathbf{I}_K\right)^{-1}\mathbb{E}[\mathbf{W}]^TX_n\\
\mathbf{S}_{Z_n} &\leftarrow \left(\frac{1}{\sigma_e^2}\mathbb{E}[\mathbf{W}^T\mathbf{W}] + \mathbf{I}_K\right)^{-1}\\
s^2_{W_{pk}} &\leftarrow \left[\frac{1}{\sigma_e^2}\left(\sum_n \mu_{Z_n,k}^2 + \mathbf{S}_{Z_n,kk}^2\right) + \frac{\psi_{pk}}{\sigma_0^2} + \frac{1 - \psi_{pk}}{\sigma_1^2}\right]^{-1}\\
\mu_{W_{pk}} &\leftarrow \frac{s^2_{W_{pk}}}{\sigma_e^2}\left[\left(\sum_n \mathbf{X}_{np}\mu_{{Z_n},k}\right) - \left(\sum_n\sum_{\ell\ne k} \mu_{W_{p\ell}}\mathbf{S}_{Z_n,k\ell}\right)\right]\\
\psi_{pk} &\leftarrow 1 - \frac{1}{1 + \frac{p_0}{1-p_0}\sqrt{\sigma_1^2/{\sigma_0^2}}\exp\left\{ \frac{1}{2}\left(\frac{1}{\sigma_0^2} - \frac{1}{\sigma_1^2}\right)(\mu_{W_{pk}}^2 + s^2_{W_{pk}})  \right\}}
\end{align*}
where
\[
\mathbb{E}[\mathbf{W}]_{pk} = \mu_{W_{pk}}
\]
and
\[
\mathbb{E}[\mathbf{W}^T\mathbf{W}]_{k\ell} = \sum_{p} \mu_{W_{pk}} \mu_{W_{p\ell}} + \delta_{k\ell} \sum_{p} s^2_{W_{pk}}.
\]

Now I derive a VI scheme using Theorems~\ref{theorem:exp_family}~and~\ref{theorem:conjugacy}.  The calculations are largely the same as in Appendix~\ref{sec:vi_ldpred} and so a number of details are omitted.  Because I am again replacing a Gaussian by a mixture of a Gaussian and point mass at zero, I assume the posterior for $\mathbf{W}_{pk}$ is a mixture of a point mass at zero and a Gaussian with mean $\mu_{W_pk}$, variance $s^2_{W_{pk}}$, and probability of being zero $\psi_{pk}$.  Working through the algebra as in the LDpred model results in:
\begin{align*}
\mu_{Z_n} &\leftarrow \frac{1}{\sigma_e^2}\left(\frac{1}{\sigma_e^2}\mathbb{E}[\mathbf{W}^T\mathbf{W}] + \mathbf{I}_K\right)^{-1}\mathbb{E}[\mathbf{W}]^TX_n\\
\mathbf{S}_{Z_n} &\leftarrow \left(\frac{1}{\sigma_e^2}\mathbb{E}[\mathbf{W}^T\mathbf{W}] + \mathbf{I}_K\right)^{-1}\\
s^2_{W_{pk}} &\leftarrow \left[\frac{1}{\sigma_e^2}\left(\sum_n \mu_{Z_n,k}^2 + \mathbf{S}_{Z_n,kk}^2\right) + \frac{1}{\sigma_1^2}\right]^{-1}\\
\mu_{W_{pk}} &\leftarrow \frac{s^2_{W_{pk}}}{\sigma_e^2}\left[\left(\sum_n \mathbf{X}_{np}\mu_{{Z_n},k}\right) - \left(\sum_n\sum_{\ell\ne k} \mu_{W_{p\ell}}(1-\psi_{p\ell})\mathbf{S}_{Z_n,k\ell}\right)\right]\\
\psi_{pk} &\leftarrow 1 - \frac{1}{1 + \frac{p_0}{1-p_0}\sqrt{\sigma_1^2 / s^2_{W_{pk}}}\exp\left\{-\frac{\mu_{W_{pk}}^2}{2s^2_{W_{pk}}}\right\}}
\end{align*}
where
\[
\mathbb{E}[\mathbf{W}]_{pk} = \mu_{W_{pk}} (1 - \psi_{pk})
\]
and
\[
\mathbb{E}[\mathbf{W}^T\mathbf{W}]_{k\ell} = \sum_{p} \mu_{W_{pk}}(1-\psi_{pk}) \mu_{W_{p\ell}}(1-\psi_{p\ell}) + \delta_{k\ell} \sum_{p} s^2_{W_{pk}} (1-\psi_{pk}).
\]

When fitting both VI schemes, I performed 250 iterations of coordinate ascent.  For the naive scheme, I first updated every coordinate of $Z$, then for each coordinate updated $Y_{pk}$ then $\mathbf{W}_{pk}$.  For the new scheme, I first updated $Z$ coordinate-wise then updated $\mathbf{W}$ coordinate-wise.  Using singular value decomposition to decompose $\mathbf{X} = \mathbf{U}\mathbf{\Sigma}\mathbf{V}^T$, I initialized $\mu_{Z_i} = \mathbf{U}_n$, $\mathbf{S}_{Z_n} = \mathbf{I}_2$ $\mu_{W_{pk}} = \mathbf{V}_{pk} \mathbf{\Sigma}_{kk}$, $s^2_{W_{pk}} = 1$ and $\psi_{pk} = 1\times10^{-10}$ for both schemes.

The updates for both models require the inversion of a $K\times K$ matrix which is $O(K^3)$ and computing $\mathbb{E}[\mathbf{W}^T\mathbf{W}]$ is $O(PK^2)$, but these can be precomputed before updating each $\mu_{Z_n}$ and $\mathbf{S}_{Z_n}$.  Then, updated each $\mu_{Z_n}$ requires $O(K^2 + PK)$ time. Therefore, updating all $\mu_{Z_n}$ and $\mathbf{S}_{Z_n}$ requires $O(NPK)$ time assuming that $K \ll N$ and $K \ll P$.  For fixed $\mu_{Z_n}$ and $\mathbf{S}_{Z_n}$, updating $\mu_{W_{pk}}$, $s^2_{W_{pk}}$, and $\psi_{pk}$ is limited by computing $\sum_n \sum_{\ell\ne k} \mu_{W_{p\ell}}\mathbf{S}_{Z_n, k\ell}$ or $\sum_n \sum_{\ell\ne k} \mu_{W_{p\ell}} (1 - \psi_{W_p\ell})\mathbf{S}_{Z_n, k\ell}$ which requires $O(NK)$ time.  Therefore updating all $\mu_{W_{pk}}$, $s^2_{W_{pk}}$, and $\psi_{pk}$ requires $O(NPK^2)$ time.  Therefore, each iteration of coordinate ascent requires $O(NPK^2)$ time.

\section{Additional PCA runs}
\label{sec:supp_ppca}
To ensure that the results presented in the main text are not unusual, I generated five additional datasets as described in the main text and compared the resulting PCA projections and sparsity of the loadings for traditional PCA (based on singular value decomposition), my naive implementation of sparse probabilistic PCA, and the implementation of sparse probabilistic PCA based on the non-overlapping mixtures trick (Figures~\ref{fig:supp_proj}~and~\ref{fig:supp_loadings}).  In all five realizations, the new formulation of sparse probabilistic PCA produces the sparsest loadings, and subjectively best separates the four clusters using the first two principle components.  As before, the naive implementation is indistinguishable from traditional PCA for small values of $\sigma_0^2$ or values of $\sigma_0^2$ close to $1$.

\begin{figure}
\centering
\includegraphics[width=0.8\textwidth]{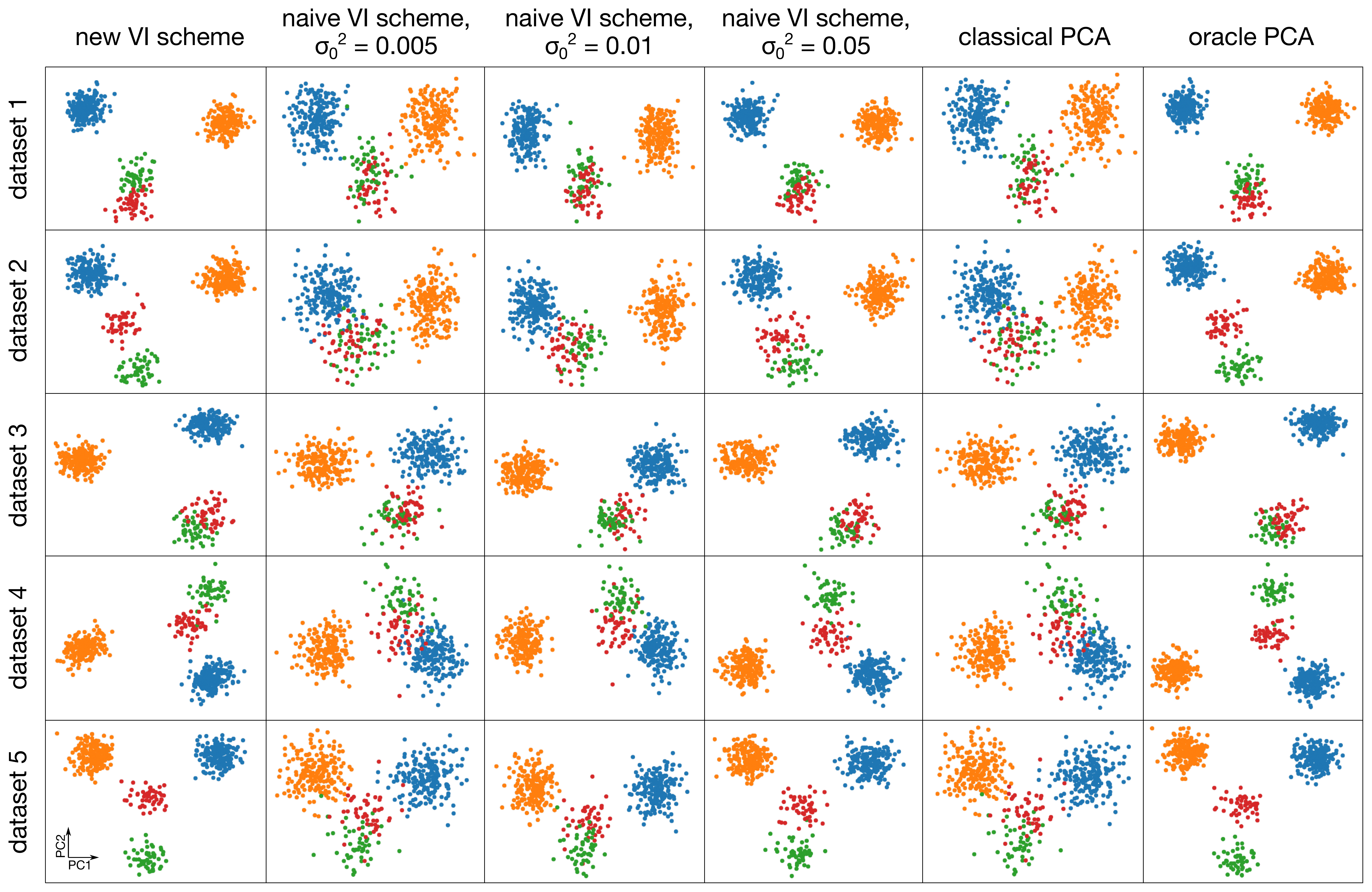}
\caption{\textbf{Projections onto the first two PCs for five replicate simulations}\\
Across all five replicate simulations the data cluster well in their projection onto the first two PCs for the new VI scheme.  While the naive scheme can cluster the data well if $\sigma^2_0$ is tuned properly, the clusters are often not as well-defined as under the new scheme.  Furthermore, the loadings are substantially less sparse as shown in Figure~\ref{fig:supp_loadings}.  In the limit of $\sigma_0^2 \downarrow 0$, it empirically appears that the naive scheme is indistinguishable from classical PCA.}
\label{fig:supp_proj}
\end{figure}
\begin{figure}
\centering
\includegraphics[width=0.6\textwidth]{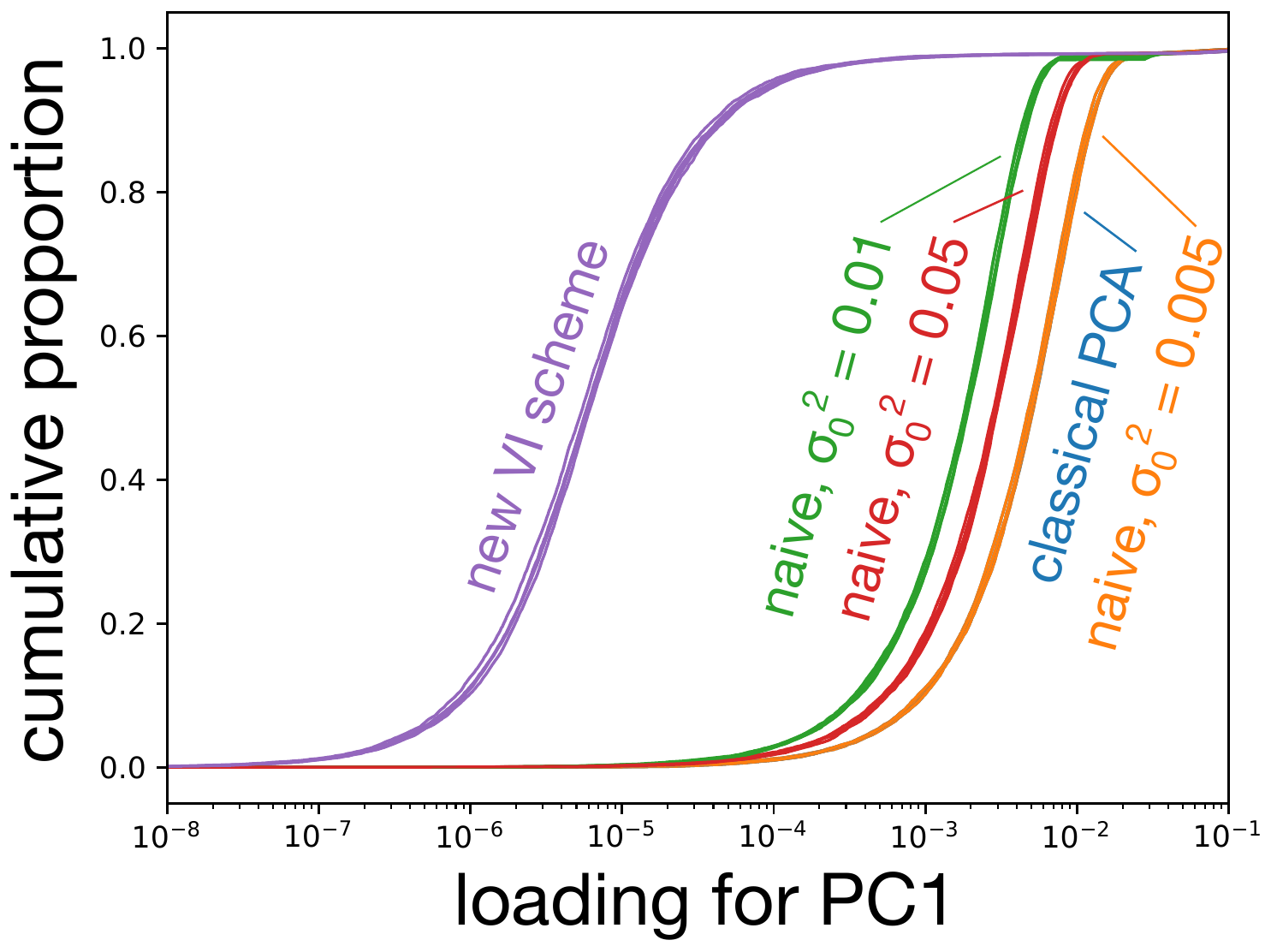}
\includegraphics[width=0.6\textwidth]{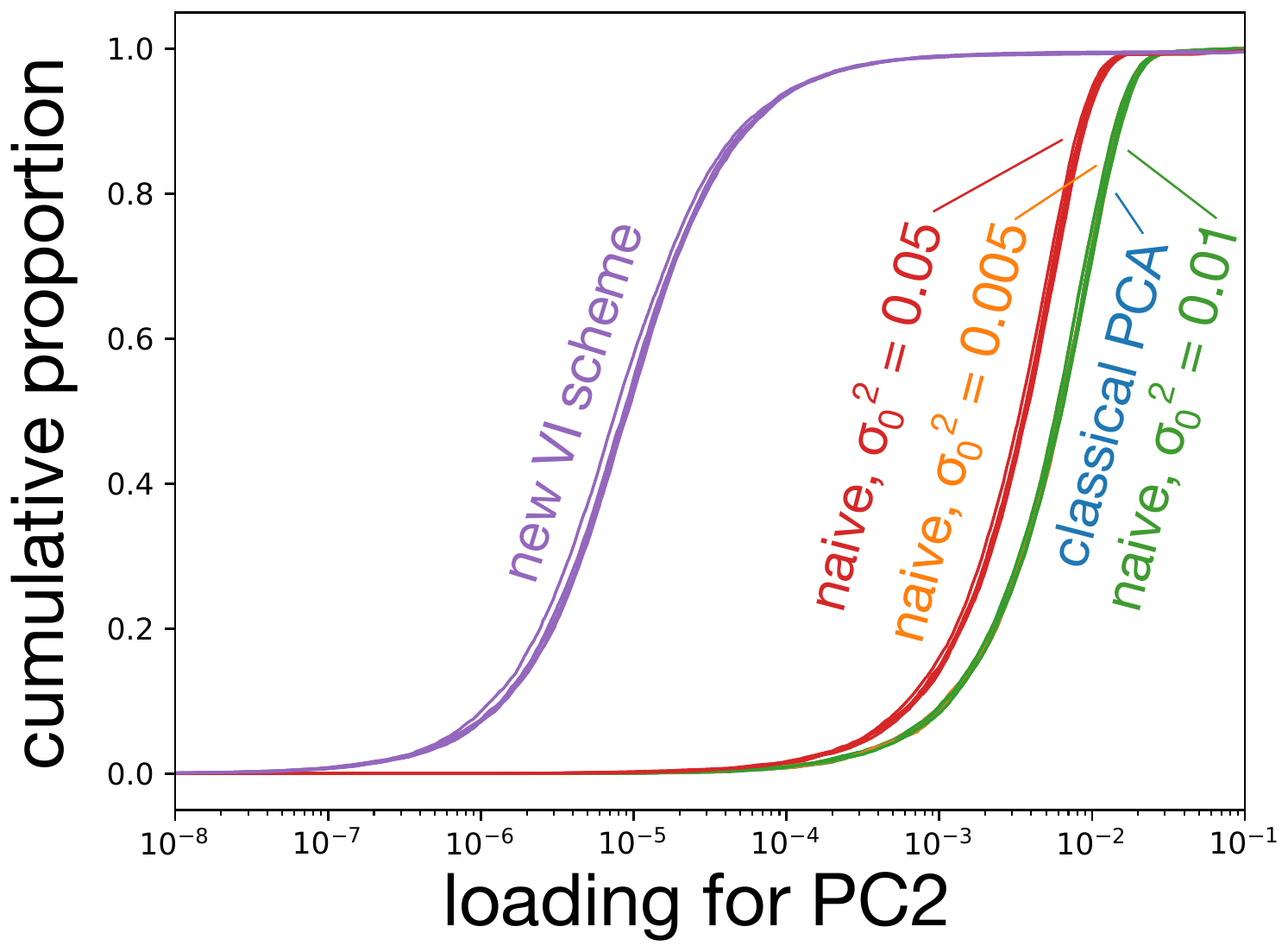}
\caption{\textbf{Distribution of loadings for five replicate simulations}\\
The new VI scheme produces significantly sparser loadings, with about 90\% of variables having an absolute loading below $1\times10^{-5}$ for both PC1 and PC2.  The naive scheme, while having a more skewed distribution of loadings than classical PCA can hardly be considered sparse with most variables having loadings greater than $1\times10^{-3}$ regardless of the precise value of $\sigma^2_0$ used.  Note that in both plots, the right-most cluster of curves is over-plotted: for PC1 the naive scheme with $\sigma^2_0=0.005$ is essentially indistinguishable from classical PCA for all five replicates.  For PC2, the naive scheme with $\sigma^2_0 = 0.005$ or $\sigma^2_0 = 0.01$ is indistinguishable from classical PCA in all but one replicate.}
\label{fig:supp_loadings}
\end{figure}

I also computed reconstruction error as a quantitative measure of performance.  I defined reconstruction error as the squared Frobenius norm between the reconstructed matrix ($\mathbb{E}[\mathbf{W}\mathbf{Z}]$ for the VI methods) and the signal in the simulated matrix -- that is, the matrix obtained by centering and scaling a matrix where each entry is the $\mu_c$ for that entry as defined above.  The mean reconstruction error across five simulations was 4261 (min: 4072, max: 4517) for the non-overlapping mixtures trick; 3985 (min: 3923, max: 4240) for oracle PCA; and 29407 (min: 28761, max: 30429) for classical PCA.  Across the naive schemes, taking $\sigma^2_0 = 0.05$ performed best with a mean reconstruction error of 7191 (min:7090, max: 7408).  Overall, the non-overlapping mixtures trick performed only slightly worse than PCA using knowledge of which variables were non-zero, whereas even the best naive scheme had almost double the reconstruction error, but all methods that attemted to account for sparsity outperformed classical PCA.

\end{document}